\documentclass[10pt,a4paper]{article}
\usepackage{amsmath,amsfonts,amsthm,amsopn,color,amssymb,graphics,graphicx}
\newcommand{\eps}{\varepsilon}

\newcommand{\R}{\mathbb{R}}
\renewcommand{\le }{\leqslant }
\renewcommand{\ge }{\geqslant }

\newtheorem{theorem}{Theorem}[section]
\newtheorem{lemma}[theorem]{Lemma}
\newtheorem{definition}[theorem]{Definition}
\newtheorem{proposition}[theorem]{Proposition}
\newtheorem{remark}[theorem]{Remark}
\newtheorem{corollary}[theorem]{Corollary}

\renewenvironment{proof}[1][Proof.]{\noindent\textbf{#1} }{$\hfill\square$\vspace{0.2 cm}\\}

\numberwithin{equation}{section}
\begin{document}

\title{{\bf{ Cracks with impedance, stable determination from boundary data} \thanks{Work supported in part by MIUR, PRIN 20089PWTPS. E.S. wishes to thank  the Istituto Nazionale di Alta Matematica (INdAM) for partly supporting her work by a research grant. Part of this work was done while the authors were attending the 2011 Programme Inverse Problems at the Isaac Newton Institute. The hospitality of the Institute is gratefully acknowledged.}}}
\author{Giovanni Alessandrini\footnote{Universit\`a degli Studi di Trieste, Italy, e-mail: \text {alessang@units.it}}  \and
Eva Sincich\footnote{Universit\`a degli Studi di Trieste, Italy,  e-mail: \text{esincich@units.it}} 
}

\date{}

\maketitle

\begin{abstract}
\noindent
We discuss the inverse problem of determining the possible presence of an $(n-1)$-dimensional crack $\Sigma$ in an $n$-dimensional body $\Omega$ with $n\ge 3$ when the so-called Dirichlet-to-Neumann map is given on the boundary of $\Omega$. In combination with quantitative unique continuation techniques,  an optimal single-logarithm stability estimate is proven by using the singular solutions method. Our arguments also apply when the Neumann-to-Dirichlet map or the local versions of the D-N and the N-D map are available.

\end{abstract}

{\small{\bf Keywords}: inverse crack problem, impedance boundary condition, stability. \ }

{\small{\bf 2000 Mathematics Subject Classification }: 35R30, 35R25,
31B20 .}

\begin{section}{Introduction}

Consider an homogeneous electrically conducting body $\Omega\subset\mathbb{R}^n$ which might contain an unknown inaccessible crack represented by an $(n-1)-$dimensional orientable surface with boundary $\Sigma \subset \subset \Omega$. Electrostatic equilibrium can be modeled by 

\begin{equation}\label{C}
\left\{
\begin{array}
{lcl}
\Delta u=0\ ,&& \mbox{in $\Omega\setminus\overline{\Sigma}$ ,}
\\
\nabla u^{\pm}\cdot{\nu^{\pm}} - \gamma^{\pm} u^{\pm} =0 \ , && \mbox{on either side of $\Sigma$ ,}
\\
 u= \varphi\ ,   && \mbox{on $\partial\Omega$ .}
\end{array}
\right.
\end{equation}
Here $u$ denotes the electrostatic potential, $\varphi$ denotes the prescribed potential distribution on the exterior boundary $\partial \Omega$. The Robin type boundary condition on $\Sigma$ has to be interpreted as follows. 

Having chosen one arbitrary orientation for the normal unit field $\nu$ on $\Sigma$ we distinguish by the $\pm$ sign the boundary values (or traces) of $u$ and its derivatives on the two sides of $\Sigma$ and we denote by $\nu^+, \nu^-$ the normals to $\Sigma$ pointing to the $+, -$ side of $\Omega\setminus\overline{\Sigma}$ respectively. 
The impedance coefficients $\gamma^+,\gamma^-$ on the two side are assumed to be non-negative. 

We consider the inverse problem of determining $\Sigma$ from boundary current density measurements $\partial_{\nu}u$ corresponding to one or more choices of the prescribed boundary potential $\varphi$.

As is well-known, since Friedman and Vogelius \cite{FV}, at least two measurements are necessary and in fact, in the two-dimensional setting, it is by now clear how two suitable boundary measurements can be chosen in order to have uniqueness and stability, \cite{A1,A2,ADV,ARo,KSe,Ro1,Ro2}. See also Bryan and Vogelius \cite{BV} for a thorough review and bibliography data. 

When the space dimension $n$ is three, or higher, uniqueness with finitely many measurements is known in few cases, we recall the paper by DiBenedetto and the first author \cite{ADiB}. Instead, uniqueness is known when full boundary data are available Eller \cite{E}, that is when the Dirichlet to Neumann map $\Lambda: \varphi \mapsto \partial_{\nu} u|_{\partial \Omega}$ is known. 

The aim of this paper is to continue the study of the $n$-dimensional problem, $n\ge 3$, initiated by DiBenedetto and A. and Eller, treating the stability issue when a full set of boundary data are available. We shall prove under some a priori regularity assumptions on the crack $\Sigma$, that the crack depends continuously upon the Dirichlet to Neumann map with a modulus of continuity of logarithmic type, with a single log. 

Note that such a modulus of continuity is in fact optimal in view of the several examples in \cite{DCRo}. Our approach is based on the use of singular solutions. This method can be traced back to Isakov \cite{Is0} and it appears also in \cite{E} in Eller's uniqueness proof. The use of singular solutions for stability estimates is by now well-established, \cite{A0,ADC,AGa1,AGa2,AKi,AVe,DC1,DC2,Sa}.
However the crack problem at hand displays several additional difficulties which have required a completely novel approach at various crucial steps.

Let us describe here the main steps of our proof with their specific difficulties.

\vskip 0.3cm 
 {\bf{First step}}
\vskip 0.3cm
For two cracks $\Sigma_1, \Sigma_2$ we consider the corresponding Dirichlet to Neumann maps $\Lambda_1, \Lambda_2$. We shall establish an identity which relates $\Lambda_1 -\Lambda_2$ with integrals on $\Sigma_1\cup\Sigma_2$ involving jumps of the corresponding potentials $u_1,u_2$ and their normal derivatives (Theorem \ref{intbyparts} below),
\begin{eqnarray}\label{identita}
&&\int_{\partial\Omega}(\Lambda_1 -\Lambda_2)u_1 u_2 d\sigma= \int_{\Sigma_1\setminus\Sigma_2}(u_2[\partial_{\nu_1}u_1]_1 - [u_1]_1\partial_{\nu_1} u_2) d\sigma +\\
&& + \int_{\Sigma_2\setminus\Sigma_1} ([u_2]_2\partial_{\nu_2}u_1- u_1[\partial_{\nu_2}u_2]_2)d\sigma + \int_{\Sigma_1\cap\Sigma_2}([u_2\partial_{\nu_1}u_1]_1 - [u_1\partial_{\nu_2}u_2]_2) d\sigma , \nonumber
\end{eqnarray} 
here $u_1$ and $u_2$ are solutions to \eqref{C} when $\Sigma=\Sigma_1=\Sigma_2$ respectively and $[\cdot]_1,[\cdot]_2$ denote jumps across $\Sigma_1,\Sigma_2$ respectively. We refer to \eqref{jump1} and \eqref{jump2} for precise definitions and to Section \ref{sectionintbyparts} for a proof. 

This identity can be viewed as the analogue for the crack problem of the so-called ``Alessandrini identity" \cite{A-1,Is1} for the Calder\'{o}n problem. However, in this case, its derivation is somewhat intricate due to the fact that the common domain of definition of $u_1$ and $u_2$ is $\Omega\setminus (\Sigma_1\cup\Sigma_2)$. Such a set, despite the regularity of $\Sigma_1$ and $\Sigma_2$, might be rather wild, thus integration by parts becomes a delicate matter, which involves also a preliminary study of the regularity of the potentials $u_1,u_2$ and the evaluation of the possible singular behavior of their gradients near the crack edges $\partial \Sigma_1, \partial \Sigma_2$, (see Theorem \ref{rego} below).   

\vskip 0.3cm 
 {\bf{Second step}}
\vskip 0.3cm
We apply the above identity to singular solutions $u_1(\cdot)=R_1(\cdot,y)$ and $u_2(\cdot)=R_2(\cdot,w)$ defined on a larger domain and each having a Green's type singularity at points $y,w$ placed outside of $\Omega$. 

Looking at the right hand side of the identity \eqref{identita} we introduce the function 
\begin{eqnarray}\label{effe}
f(y,w)&=&\int_{\Sigma_1\setminus\Sigma_2}(R_2(\cdot,y)[\partial_{\nu_1}R_1(\cdot,w)]_1 - [R_1(\cdot,y)]_1\partial_{\nu_1} R_2(\cdot,w)) d\sigma+\ \ \ \nonumber\\
&&+\int_{\Sigma_2\setminus\Sigma_1} ([R_2(\cdot,w)]_2\partial_{\nu_2}R_1(\cdot,y)- R_1(\cdot,y)[\partial_{\nu_2}R_2(\cdot,w)]_2)d\sigma+\ \ \  \nonumber\\
&&+\int_{\Sigma_1\cap\Sigma_2}([R_2(\cdot,w)\partial_{\nu_1}R_1(\cdot,y)]_1 - [R_1(\cdot,y)\partial_{\nu_2}R_2(\cdot,w)]_2) d\sigma .\ \ \ \ 
\end{eqnarray}
Note that $f$ is harmonic in $\Omega\setminus (\Sigma_1\cup\Sigma_2)$ separately in each variable $y,w$. Moreover $f$ is controlled in terms of $\Lambda_1 -\Lambda_2$ when $y,w$ are outside $\Omega$. 

Next, by estimates of propagation of smallness for harmonic functions we are able to bound $f(y,w)$ when $y=w$ approaches points of $\Sigma_1 \triangle \Sigma_2$ (the symmetric difference). Here the technical obstruction come from the fact that propagation of smallness can be performed only on connected sets, whereas $\Omega\setminus (\Sigma_1\cup\Sigma_2)$ may be not. Moreover, not all points of $\Sigma_1 \triangle \Sigma_2$ may be reachable from the exterior of $\Omega$ and the estimates of propagation of smallness require that points be reachable in a quantitative form which involves the use of chains of balls  whose numbers is suitably bounded and their radii have to be bounded from below. Such requirements induce the introduction of sets $V_l$ of points which can be suitably reached from the exterior of $\Omega$ (see definition \eqref{Vl}) and an ad-hoc definition of a variation of the Hausdorff distance for closed sets which we call $l-$distance. 

The crucial point here is that under the a priori regularity assumptions on $\Sigma_1,\Sigma_2$ we can show that the Hausdorff distance is dominated by the respective $l-$distance (see Proposition \ref{equiv}).

\vskip 0.3cm 
 {\bf{Third step}}
\vskip 0.3cm
We show that as $y=w$ tends to a point of $\Sigma_1\triangle\Sigma_2$ then $f(y,y)$ blows up. The combination of such a blow up bound and the estimate of $f(y,y)$ in terms of $\Lambda_1 - \Lambda_2$ obtained in the previous step lead to the logarithmic estimate of $d_{H}(\sigma_1, \Sigma_2)$ in terms of $\|\Lambda_1-\Lambda_2\|$. The blow up estimate of this step requires a careful investigation of the asymptotic behavior of the singular solutions $R_i(\cdot, y)$ as their pole $y$ approaches to the crack $\Sigma_i, i=1,2$ (see Proposition \ref{stimeasintotiche} and Proposition \ref{holder}).

\end{section}
\begin{section}{The main results}\label{assumptions}
\begin{subsection}{Notation and definitions}
\end{subsection}
In the sequel, we shall make a repeated use of quantitative notions of smoothness for the boundary of the domain $\Omega$ and for the crack $\Sigma$. Let us introduce the following notation and definitions.   

In several places it will be useful to single out one coordinate direction, to this purpose, we shall use the following notions for points $x\in \mathbb{R}^n, x'\in\mathbb{R}^{n-1}, x''\in\mathbb{R}^{n-2}, n\ge 3, x=(x',x_n), x'=(x'',x_{n-1}), x''= (x''', x_{n-2})$, with $x'\in \mathbb{R}^{n-1}, x''\in \mathbb{R}^{n-2}, x'''\in \mathbb{R}^{n-3}$ and $x_n, x_{n-1}, x_{n-2}\in \mathbb{R}.$ Moreover, given a point $x\in \mathbb{R}^n$, we shall denote with $B_r(x), B_r'(x), B_r^{''}(x), B_r^{'''}(x)$ the ball in $\mathbb{R}^{n},\mathbb{R}^{n-1},\mathbb{R}^{n-2}, \mathbb{R}^{n-3}$ respectively centered in $x$ with radius $r$.
 \begin{definition}
Let $\Omega$ be a domain in $\mathbb R^n$. 
We say that $\partial\Omega$
is of class $C^{0,1}$ with constants $r_0,M$ if for any $P\in\partial\Omega$ there exists a rigid transformation
of $\mathbb R^n$ under which we have $P\equiv0$ and
$$\Omega\cap B_{r_0}=\{x\in B_{r_0}\,:\,x_n>\varphi(x')\},$$
where $\varphi$ is a $C^{0,1}$ function on $B'_{r_0}$ satisfying the following condition
$\varphi(0)=|\nabla_{x'}\varphi(0)|=0$ and $\|\varphi\|_{C^{0,1}(B'_{r_0})}\leq Mr_0$, 
where we denote
\begin{eqnarray}\label{chap2:8}
\|\varphi\|_{C^{0, 1}( B^{'}_{r_0})}&=&\|\varphi\|_{L^{\infty}( B^{'}_{r_0})}+ {r_0}\sup_{\substack {x,y  \in B^{'}_{r_0}\\x\ne y }}\frac{|\varphi (x)-\varphi (y)|}{|x-y|}\  .\nonumber
\end{eqnarray}

\end{definition}
\begin{definition}
Given $ \alpha,\ 0<\alpha\le 1$, we shall say that an hypersurface $S$ is of \emph{class $C^{1,\alpha}$ with constants} $r_0,\ M>0$ 
if for any $P \in S$, there exists a rigid transformation of coordinates under which we have $P=0$ and 
\begin{eqnarray}\label{chap2:1}
 S \cap B_{r_0}=\{(x',x_n)\in B_{r_0}: x_n=\varphi(x')\}\ 
\end{eqnarray}
where
\begin{eqnarray}\label{chap2:2}
\varphi:B^{'}_{r_0}\subset \R^{n-1} \rightarrow \mathbb{R} 
\end{eqnarray}
 is a  $C^{1,\alpha}$ function satisfying 
\begin{eqnarray}\label{chap2:25}
|\varphi(0)|=|\nabla \varphi(0)|=0\ \ \mbox{and}\ \ \|\varphi\|_{C^{1,\alpha}(B^{'}_{r_0})}\le Mr_0\ ,
\end{eqnarray}
where we denote
\begin{eqnarray}\label{chap2:3}
\|\varphi\|_{C^{1, \alpha}( B^{'}_{r_0})}&=&\|\varphi\|_{L^{\infty}( B^{'}_{r_0})}+ r_0\|\nabla\varphi\|_{L^{\infty}( B^{'}_{r_0})} + \\
&+&{r_0}^{1+\alpha}\sup_{\substack {x,y  \in B^{'}_{r_0}\\x\ne y }}\frac{|\nabla\varphi (x)-\nabla \varphi (y)|}{|x-y|^{\alpha}}\  .\nonumber
\end{eqnarray}
\end{definition}


We introduce some notations that we shall use in the sequel.

For any $0<r<r_0$ and any $0<r_1<r_2<r_0$ we shall denote
\begin{eqnarray}
&&\Sigma_{r}=\{x\in \Sigma : \mbox{dist}(x, \partial \Sigma)>r \} \ , \\
&&E_{r}=\{x\in \mathbb{R}^n: \mbox{dist}(x,\Sigma)>r \}\ ,\\
&&\mathcal{U}^i_{r}= \{x\in \mathbb{R}^n \ : \ \mbox{dist}(x,\partial \Sigma_i)<r \}\ , \\
&&\Omega_r=\{ x\in \Omega^c \ : \ \mbox{dist}(x,\Omega)\le r  \}\ ,\\
&&\mathcal{S}_{r_1,r_2}=\{ x\in \mathbb{R}^n : r_1\le \mbox{dist}(x,\partial \Omega)\le r_2 \}\ ,\\
&&\Gamma_r=\{x \in \mathbb{R}^n: \mbox{dist}(x,\Omega)=r \} .
\end{eqnarray}

\begin{subsection}{The D-N map}

We begin by defining the Dirichlet to Neumann map.

For any $\varphi\in H^{\frac{1}{2}}(\partial\Omega)$, the unique weak solution to the mixed Dirichlet-Robin type problem 

\begin{equation}\label{C1}
\left\{
\begin{array}
{lcl}
\Delta u=0\ ,&& \mbox{in $\Omega\setminus\overline{\Sigma}$ ,}
\\
 u= \varphi\ ,   && \mbox{in the trace sense on $\partial\Omega$ ,}
\\
\partial_{\nu^{\pm}}u^{\pm} - \gamma^{\pm} u^{\pm}=0 \ , && \mbox{on $\Sigma^{\pm}$ ,}
\end{array}
\right.
\end{equation}
is given as the unique minimizer of the quadratic form
\begin{eqnarray}
Q_{\Sigma}(u)=\int_{\Omega}|\nabla u|^2 + \int_{\Sigma}\gamma^+{u^+}^2 + \gamma^-{u^-}^2
\end{eqnarray}
among all $u \in H^1(\Omega\setminus \Sigma),\ \ u|_{\partial \Omega}=\varphi$.    

We denote by $<\cdot,\cdot>$ the $L^2(\partial \Omega)$ pairing between $H^{\frac{1}{2}}(\partial \Omega)$ and $H^{-\frac{1}{2}}(\partial \Omega)$.

\begin{definition}\label{DtN}
The Dirichlet to Neumann map associated to \eqref{C1} is the operator
\begin{eqnarray}
\Lambda:H^{\frac{1}{2}}(\partial \Omega)\rightarrow H^{-\frac{1}{2}}(\partial \Omega)
\end{eqnarray}
defined 
by 
\begin{eqnarray}
<\Lambda\varphi, \eta>=\int_{\Omega}\nabla v\cdot \nabla u + \int_{\Sigma}\gamma^+v^+u^+ + \gamma^-v^-u^-
\end{eqnarray} 
for every  $\varphi, \eta\in H^{\frac{1}{2}}(\partial\Omega)$ where $u$ is the solution to \eqref{C1} and $v\in H^1(\Omega\setminus\overline{\Sigma})$ is such that $v|_{\partial \Omega}=\eta$. 
\end{definition}
Note that, as an immediate consequence, we deduce that
\begin{eqnarray}
\Lambda:H^{\frac{1}{2}}(\partial \Omega)\rightarrow H^{-\frac{1}{2}}(\partial \Omega)
\end{eqnarray}
is selfadjoint.
\end{subsection}

\begin{subsection}{Assumptions and a-priori information}
{\bf{Assumption on the domain}}

Given $r_0,M,D>0$ constants we assume that $\Omega \subset \R^{n}$ and 
\begin{eqnarray}\label{domain}
\Omega \   \mbox{is of} \  C^{0,1}\ \mbox{class with constants}\ r_0, M
\end{eqnarray}
such that $\partial \Omega$ is connected.
Furthermore, $\Omega$ is such that 
\begin{eqnarray}\label{diametro}
\mbox{diam}(\Omega)\le D
\end{eqnarray} 
Moreover, we assume that the crack $\Sigma$ is contained into a closed connected hypersurface $\Gamma\subset\Omega$ such that
\begin{eqnarray}\label{superficie}
\Gamma\ \ \mbox{is}\ C^{1,\alpha} \ \mbox{smooth with constants}\ r_0, M 
\end{eqnarray}
and it diffeomorphic to a sphere. 
We also suppose that 
\begin{eqnarray}
\Sigma\  \mbox{within}\ \Gamma\ \mbox{is of class}\ C^{1,\alpha}\  \mbox{with constants}\ r_0, \ M.
\end{eqnarray}
 Namely, for any $Q\in \partial\Sigma$, there exists a rigid transformation of coordinates under which we have $Q=0$ and 
\begin{eqnarray}\label{tip}
 \Sigma\cap B_{r_0}=\{(x',x_n)\in B_{r_0}: x_n=\varphi(x'),\ x_{n-1}>\psi(x'')\}\ 
\end{eqnarray}
where
\begin{eqnarray}\label{psi}
\psi:B^{''}_{r_0}\subset \R^{n-2} \rightarrow \mathbb{R} 
\end{eqnarray}
satisfying
\begin{eqnarray}\label{psi2}
\psi(0)=|\nabla\psi(0)|=0 \ \ \ \mbox{and}\  \ \ \|\psi\|_{C^{1,\alpha}}\le M.
\end{eqnarray}

{\bf{Assumptions on the crack impedances}}

Given a positive number $\overline\gamma$,
the crack impedances $\gamma^+$ and $\gamma^-$ of the unknown crack $\Sigma$ are such that
\begin{subequations}
\label{gamma}
\begin{equation}
\label{gamma0}
\gamma^{\pm}\in C^{0,1}(\Sigma)
\end{equation}
and 
\begin{equation}
\label{040}
0\leq\gamma^{\pm}(x)\leq\overline{\gamma}\ \ \mbox{for any}\ \ x\in \Sigma.
\end{equation}
\end{subequations}

We shall refer to the $r_0,M,D,\bar{\gamma}$ along with the space dimension $n$ as to the a priori data. 
\end{subsection}
\begin{subsection}{The main results}

We start by collecting our main stability results for the unknown crack and the unknown impedance by means of the global D-N map. 
\begin{theorem}\label{main}
Let $\Omega,\Sigma_1,\Sigma_2$ be the domain and the cracks satisfying the a-priori assumptions stated above. If, given $\eps>0$, we have that the D-N maps $\Lambda_1$ and $\Lambda_2$ corresponding to the cracks $\Sigma_1$ and $\Sigma_2$ respectively, satisfy
\begin{eqnarray}\label{errore}
\|\Lambda_1 -\Lambda_2\|_{\mathcal{L}(H^{\frac{1}{2}}(\partial \Omega), H^{-\frac{1}{2}}(\partial \Omega))}\le \eps
\end{eqnarray}
then 
\begin{eqnarray}
d_{H}(\Sigma_1,\Sigma_2)\le C|\log(\eps)|^{-\eta}
\end{eqnarray}
where $C,\eta>0$ are constants depending on the a-priori data only. 
\end{theorem}

\begin{corollary}
Under the same hypothesis of Theorem \ref{main}, we have also that  
\begin{eqnarray*}  
\sup\{ |\gamma_2^{\pm}(Q)-\gamma_1^{\pm}(P)| \ \mbox{s.t.} \  P \in {\Sigma_1}^{r_0}, Q\in {\Sigma_2}^{r_0}\cap B_{2 C|\log(\eps)|^{-\eta}}(P) \}\le C'|\log(\eps)|^{-\eta'} 
\end{eqnarray*}
where $C',\eta'>0$ are constants depending on the a priori data only. 
\end{corollary}

\begin{proof}
The Corollary follows by combining the result in Theorem \ref{main} and quantitative stability estimates for the Cauchy problem. 
For the details of the proof we refer to \cite[Theorem 2.3]{Si}. 
\end{proof}
\end{subsection}

\begin{subsection}{Variants}
In addition, we now state some variants of Theorem \ref{main} basically relying on other types of data availability. We shall omit proofs since they require only minimal adjustments in comparison to the proof of Theorem \ref{main}.

We start by defining the local version of the Dirichlet to Neumann map. 

Let us fix an open neighborhood $\Delta_{{\rho}_0}=B_{{\rho}_0}(x_0)\cap\partial \Omega$ for a fixed point $x_0\in \partial \Omega$ and a given $\rho_0>0$ . We introduce the trace space $H_{00}^{\frac{1}{2}}(\Delta_{{\rho}_0})$  as the interpolation space $[H^{1}_0(\Delta_{{\rho}_0}),L^2(\Delta_{{\rho}_0})]_{\frac{1}{2}}$, we refer to \cite[Chap.1]{LiMa} for further details . The functions in $H_{00}^{\frac{1}{2}}(\Delta_{{\rho}_0})$ might be also characterized as the elements in $H^{\frac{1}{2}}(\partial \Omega)$ which are identically zero outside $\Delta_{{\rho}_0}$ (see for instance \cite{T}), this identification shall be understood throughout. We denote with $H_{00}^{-\frac{1}{2}}(\Delta_{{\rho}_0})$ its dual space, which also can be interpreted as a subspace of $H^{-\frac{1}{2}}(\partial \Omega)$.  We continue to use the notation $<\cdot,\cdot>$ for the duality pairing between $H_{00}^{\frac{1}{2}}(\Delta_{{\rho}_0})$ and $H_{00}^{-\frac{1}{2}}(\Delta_{{\rho}_0})$ based on the $L^2$ scalar product. 

\begin{definition}
 We shall define as the local Dirichlet to Neumann map associated to \eqref{C1} and $\Delta_{{\rho}_0}$ the operator 
\begin{eqnarray}
\Lambda^{\Delta_{{\rho}_0}}:H_{00}^{\frac{1}{2}}(\Delta_{{\rho}_0})\rightarrow H_{00}^{-\frac{1}{2}}(\Delta_{{\rho}_0})
\end{eqnarray}
defined again by 
\begin{eqnarray}
<\Lambda^{\Delta_{{\rho}_0}}\varphi, \eta>=\int_{\Omega}\nabla u\cdot \nabla v  + \int_{\Sigma}\gamma^+u^+v^+ + \gamma^-u^-v^-
\end{eqnarray} 
for every $\varphi, \eta \in H_{00}^{\frac{1}{2}}(\Delta_{\rho_0})$ where $u$ is the solution to \eqref{C1} and $v\in H^1(\Omega\setminus\overline{\Sigma})$ is such that $v|_{\partial \Omega}=\eta$.

\end{definition}

We now consider the global Neumann to Dirichlet map and we introduce the following space of distributions ${}_0H^{-\frac{1}{2}}(\partial \Omega)=\{\eta\in H^{-\frac{1}{2}}(\partial \Omega)\ : \ <\eta,1> =0 \}$ .
\begin{definition}
We refer to the Neumann to Dirichlet map as to the selfadjoint operator 
\begin{eqnarray}
N : {}_0H^{-\frac{1}{2}}(\partial \Omega) \rightarrow H^{\frac{1}{2}}(\partial \Omega)
\end{eqnarray}
such that 
\begin{eqnarray}
<\eta, N \eta>= \int_{\Omega}|\nabla u|^2 + \int_{\Sigma}\gamma^+{u^+}^2 + \gamma^-{u^-}^2
\end{eqnarray}
for any $\eta \in {}_0H^{-\frac{1}{2}}(\partial \Omega) $, where $u\in H^1(\Omega\setminus \Sigma)$ is the weak solution to the mixed Neumann-Robin type problem
\begin{equation}\label{C2}
\left\{
\begin{array}
{lcl}
\Delta u=0\ ,&& \mbox{in $\Omega\setminus\overline{\Sigma}$ ,}
\\
\partial_{\nu} u= \eta\ ,   && \mbox{on $\partial\Omega$ ,}
\\
\partial_{\nu^{\pm}}u^{\pm} - \gamma^{\pm} u^{\pm}=0 \ , && \mbox{on $\Sigma^{\pm}$ .}
\end{array}
\right.
\end{equation}
If, $\gamma^+\equiv\gamma^-\equiv 0$ on $\Sigma^{\pm}$, we additionally require in \eqref{C2} the normalization condition $\int_{\partial \Omega} u = 0$.

\end{definition}
We are now finally in position to deal with the local Neumann to Dirichlet map. 
Denote $\Delta_{{\rho}_0}^{\prime} =\partial \Omega \setminus \overline{\Delta_{{\rho}_0}}$.
Let us consider the following space of distributions ${}_0H^{-\frac{1}{2}}(\Delta_{{\rho}_0})=\{\eta \in {}_0H^{-\frac{1}{2}}(\partial \Omega)\ :\ <\eta,\varphi>=0 \ \forall \varphi \in  H_{00}^{\frac{1}{2}}(\Delta_{\rho_0}^{\prime})\}$.
\begin{definition}
We shall define as the local Neumann to Dirichlet map associated to $\Delta_{{\rho}_0}$ the operator
\begin{eqnarray}
N^{\Delta_{{\rho}_0}} : {}_0H^{-\frac{1}{2}}(\Delta_{{\rho}_0}) \rightarrow \left({}_0H^{-\frac{1}{2}}(\Delta_{{\rho}_0})\right)^{\ast} \subset H^{\frac{1}{2}}(\partial\Omega)
\end{eqnarray}
such that 
\begin{eqnarray}
<\eta, N^{\Delta_{{\rho}_0}} \eta>= \int_{\Omega}|\nabla u|^2 + \int_{\Sigma}\gamma^+{u^+}^2 + \gamma^-{u^-}^2
\end{eqnarray}
for any $\eta \in {}_0H^{-\frac{1}{2}}(\Delta_{{\rho}_0}) $, where $u\in H^1(\Omega\setminus \Sigma)$ is the weak solution to \eqref{C2}. 
Again, if $\gamma^+\equiv\gamma^-\equiv 0$ on $\Sigma^{\pm}$, we further impose the condition $\int_{\partial \Omega} u = 0$.

\end{definition}

The first variants of our main result concerns the case when the Neumann to Dirichlet map is at our disposal instead.

\begin{theorem}\label{mainN}
Let the hypothesis of Theorem \ref{main} be fulfilled. If, given $\eps>0$, we have that the N-D maps $N_1$ and $N_2$ corresponding to the cracks $\Sigma_1$ and $\Sigma_2$ respectively, satisfy
\begin{eqnarray}\label{erroreN}
\|N_1 -N_2\|_{\mathcal{L}({}_0H^{-\frac{1}{2}}(\partial \Omega), H^{\frac{1}{2}}(\partial \Omega))}\le \eps
\end{eqnarray}
then 
\begin{eqnarray}
d_{H}(\Sigma_1,\Sigma_2)\le C|\log(\eps)|^{-\eta}
\end{eqnarray}
where $C,\eta>0$ are constants depending on the a-priori data only. 
\end{theorem}
Finally, we treat the cases when the measurements can be performed only on an open, non-empty subset $S$ of $\partial \Omega$. Such an instance leads to the introduction of the local D-N map and the local N-D map.

\begin{theorem}\label{mainl}
Let $\Omega,\Sigma_1,\Sigma_2$ be the domain and the cracks satisfying the a-priori assumptions stated above. If, given $\eps>0$, we have that the local D-N maps $\Lambda^{\Delta_{{\rho}_0}}_1$ and $\Lambda^{\Delta_{{\rho}_0}}_2$ associated to $\Delta_{{\rho}_0}$ and corresponding to the cracks $\Sigma_1$ and $\Sigma_2$ respectively, satisfy
\begin{eqnarray}\label{errorel}
\|\Lambda^{\Delta_{{\rho}_0}}_1 -\Lambda^{\Delta_{{\rho}_0}}_2\|_{\mathcal{L}(H_{00}^{\frac{1}{2}}(\Delta_{{\rho}_0}), H_{00}^{-\frac{1}{2}}(\Delta_{{\rho}_0}))}\le \eps
\end{eqnarray}
then 
\begin{eqnarray}
d_{H}(\Sigma_1,\Sigma_2)\le C|\log(\eps)|^{-\eta}
\end{eqnarray}
where $C,\eta>0$ are constants depending on the a-priori data only. 
\end{theorem}

\begin{theorem}\label{mainNl}
Let the hypothesis of Theorem \ref{main} be fulfilled. If, given $\eps>0$, we have that the local N-D maps $N^{\Delta_{{\rho}_0}}_1$ and $N^{\Delta_{{\rho}_0}}_2$ associated to $\Delta_{{\rho}_0}$ and corresponding to the cracks $\Sigma_1$ and $\Sigma_2$ respectively, satisfy
\begin{eqnarray}\label{erroreNl}
\|N^{\Delta_{{\rho}_0}}_1 -N^{\Delta_{{\rho}_0}}_2\|_{\mathcal{L}({}_0H^{-\frac{1}{2}}(\Delta_{{\rho}_0}),  \left({}_0H^{-\frac{1}{2}}(\Delta_{{\rho}_0})\right)^{\ast})}\le \eps
\end{eqnarray}
then 
\begin{eqnarray}
d_{H}(\Sigma_1,\Sigma_2)\le C|\log(\eps)|^{-\eta}
\end{eqnarray}
where $C,\eta>0$ are constants depending on the a-priori data only. 
\end{theorem}

The proofs of the last two theorems can be achieved by combining the results in Theorem \ref{main} and in Theorem \ref{mainN} respectively with the arguments in \cite{AKi} where the authors provided a quite general method which allow to obtain an H\"{o}lder type dependence of a global D-N map from a local one in a larger domain (see also \cite{AGa2,AVe} for related results).

Of course, a more general portion $\mathcal{U}$ of $\partial \Omega$ could be used in the above theorems with local data. However the stability constants shall necessarily depends on the inradius of such a portion $\mathcal{U}$. For this reason, there is no loss of generality, in formulating the above theorems in terms of the spherical neighborhood $\Delta_{{\rho}_0}$.

\begin{remark}
For the sake of brevity we only discuss here the stability issue for the $n$-dimensional case with $n\ge 3$. However our arguments and our results could be adapted to the $2$- dimensional setting. 
\end{remark}
\end{subsection}
\end{section}

\begin{section}{The direct problem}\label{directproblem}
We begin our analysis of the direct problem by providing two results of regularity near the crack for the solution to \eqref{C} near the crack, which are collected in the Theorem below and whose proof will be provided in Section \ref{regol}.

\begin{theorem}\label{rego}
Let $u$ be a solution to \eqref{C}, then there exist constants $C>0$ and $\alpha$ with $0<\alpha<1$ depending on the a priori data only such 
\begin{eqnarray}
\|u\|_{C^{0,\alpha}(\Sigma)}\le C \ .
\end{eqnarray}
Moreover, for any $\rho \in (0,r_0)$ there exists a constant $C_{\rho}>0$ depending on $\rho$ and on the a priori data only such that
\begin{eqnarray}
\|u\|_{C^{1,\alpha}(\Sigma_{\rho})}\le C_{\rho}\ .
\end{eqnarray}
\end{theorem}


As next step, in the preliminary direct problem treatment we derive an integration by parts formula for solutions to the crack problem at hand \eqref{C}.

Let $\Gamma_i,\ i=1,2$ be two closed connected orientable hypersurfaces of class $C^{1,\alpha}$ as in Section \ref{assumptions}. Just for simplicity of exposition we assume that they are diffeomorphic to a sphere. 

By the Jordan separation theorem, $\Gamma_i$ disconnects $\mathbb{R}^n$ into two connected components $\Omega_i^-, U_i$, the first one being bounded and the second one unbounded. Being $\partial \Omega$ connected and $\Omega$ bounded, we have $\Omega_i^-\subset\subset \Omega$ and $\partial \Omega\subset U_i$. We denote $\Omega^+_i=U_i\cap \Omega$. Denote $\nu_i$ the unit normal on $\Gamma_i$ pointing to its exterior $\Omega_i^+$. Furthermore, the exterior normal to $\partial \Omega$ will be denoted by $\nu_{e}$ (or simply $\nu$). 

Let $v\in H^1(\Omega\setminus\Gamma_i)$. We denote $t_i^{\pm}v $ the $H^{\frac{1}{2}}$ traces of $v$ on the two sides of $\Gamma_i$. Namely, $t_i^{\pm}$ is the trace on $\Gamma_i$ of $v|_{\Omega_i^{\pm}}$. 

We shall introduce also the jump of the traces on $\Gamma_i$ as follows 
\begin{eqnarray}\label{jump1}
[v]_i= t_{i}^+v - t_{i}^-v .
\end{eqnarray}

If, in addition, we have $\Delta v \in L^2(\Omega\setminus\Gamma_i)$ then also the one-sided normal derivatives $\partial^+_{\nu_i}v, \partial^-_{\nu_i}v, \partial_{\nu_e}v$ are defined in the distributional sense. 

We also define 
\begin{eqnarray}\label{jump2}
[\partial_{\nu}v]_i= \partial^+_{\nu_i}v - \partial^-_{\nu_i}v .
\end{eqnarray}

\begin{theorem}\label{intbyparts}{(\bf{The integration by parts formula})}

Let $u_i\in H^1(\Omega\setminus\overline{\Sigma}_i)$ be the solution to the problem \eqref{C} with $\Sigma=\Sigma_i \ i=1,2$. Then, the following identity holds
\begin{eqnarray}\label{intparts}
&&\int_{\partial\Omega}(\Lambda_1 -\Lambda_2)u_1 u_2 d\sigma= \int_{\Sigma_1\setminus\Sigma_2}(u_2[\partial_{\nu_1}u_1]_1 - [u_1]_1\partial_{\nu_1} u_2) d\sigma +\\
&& + \int_{\Sigma_2\setminus\Sigma_1} ([u_2]_2\partial_{\nu_2}u_1- u_1[\partial_{\nu_2}u_2]_2)d\sigma + \int_{\Sigma_1\cap\Sigma_2}([u_2\partial_{\nu_1}u_1]_1 - [u_1\partial_{\nu_2}u_2]_2) d\sigma . \nonumber
\end{eqnarray} 
\end{theorem}

The proof shall be given in Section \ref{sectionintbyparts}.
 
\begin{remark}
Note that the integral on the left hand side of \eqref{intparts} should  be properly interpreted as $<(\Lambda_1-\Lambda_2)u_1,u_2>$. Also , if $\eta_i=\Lambda_i u_i, \ i=1,2$, we also have that the left hand side can be written as $<\eta_1, (N_2-N_1)\eta_2>$ .

\end{remark}

\end{section}
\begin{section}{Singular solutions}\label{singularsolutions}
In this section we shall discuss and state the upper bound and the lower bound for the function $f$ introduced in \eqref{effe} and we shall obtain our main result as a combination of the two latter bounds.

We begin by introducing the so called Robin function. 

Fix $\widetilde{\Omega}$ such that $\Omega \subset\subset \widetilde{\Omega}$, we shall denote with $R$ the Robin function (or Green's function of third kind) associated to the problem \eqref{C}.

\begin{equation}\label{R}
\left\{
\begin{array}
{lcl}
\Delta_x R(x,y)=- \delta(x-y)\ ,&& \mbox{in $\widetilde{\Omega}\setminus\Sigma$ ,}
\\
\partial_{\nu^{\pm}}R^+(\cdot, y) - \gamma^{\pm}(\cdot) R^{\pm}(\cdot, y) =0 \ , && \mbox{on $\Sigma^{\pm}$ ,}
\\
\partial_{\nu}R(\cdot,y)= - \frac{1}{|\partial \widetilde{\Omega}|}\ ,   && \mbox{on $\partial \widetilde{\Omega}$ ,}
\end{array}
\right.
\end{equation}
with $y\in \widetilde{\Omega}\setminus \Sigma$.

We shall denote with $R_1$ and $R_2$ the Robin functions solutions to \eqref{R} when $\Sigma$ is replaced by $\Sigma_1$ and $\Sigma_2$ respectively. 

Let us now define, for $y,w\in \widetilde{\Omega}\setminus \Sigma$ 
\begin{eqnarray}
S_{\Sigma_1}(y,w)=&&\int_{\Sigma_1\setminus\Sigma_2}\left(R_2(\cdot,w)[\partial_{\nu_1}R_1(\cdot,y)]_1 - [R_1(\cdot,y)]_1\partial_{\nu_1} R_2(\cdot,w)\right) d\sigma +\nonumber\\
&&+ \int_{\Sigma_1\cap\Sigma_2}[R_2(\cdot,w)\partial_{\nu_1}R_1(\cdot,y)]_1  d\sigma 
\end{eqnarray}
\begin{eqnarray}
S_{\Sigma_2}(y,w)=&&\int_{\Sigma_2\setminus\Sigma_1}\left(R_1(\cdot,y)[\partial_{\nu_2}R_2(\cdot,w)]_2 - [R_2(\cdot,w)]_2\partial_{\nu_2} R_1(\cdot,y)\right) d\sigma +\nonumber\\
&&+ \int_{\Sigma_1\cap\Sigma_2}[R_1(\cdot,y)\partial_{\nu_2}R_2(\cdot,w)]_2  d\sigma 
\end{eqnarray}
note that clearly we have 
\begin{eqnarray}\label{eg}
f(y,w)= S_{\Sigma_1}(y,w) -  S_{\Sigma_2}(y,w) \ .
\end{eqnarray}
By Theorem \ref{intbyparts} we have that for every $y,w\in \widetilde{\Omega}\setminus \overline{\Omega}$ 
\begin{eqnarray}\label{eguaglianzafond}
f(y,w)= \int_{\partial \Omega} (\Lambda_1 -\Lambda_2)R_1(\cdot,y)R_2(\cdot,w)d\sigma \ .
\end{eqnarray}

\begin{subsection}{Upper bound on the function f}
Given $A,l>0$ we consider the cone 
\begin{eqnarray}
C_l=\{x=(x',x_n): 0<x_n<Al\ , |x'|+\frac{x_n}{A}<l  \}
\end{eqnarray}
and for any orthogonal transformation $R$ and any point $z$, we denote with 
\begin{eqnarray}
RC_l(z)=RC_l + z \ ,
\end{eqnarray}
the rotated cone whose basis is centered in $z$.

Given $\gamma :[0,1]\rightarrow \Omega \cup \Omega_r$ a simple arc, we define the following set 
\begin{eqnarray}
\gamma^l=\bigcup_{\ t\in [0,1]} B_l(\gamma(t))\cup RC_l(\gamma(1)) \ .
\end{eqnarray}
Denoting with $P(\gamma^l)$ the vertex of the cone $RC_l(\gamma(1))$ and given $0<r<r_0$, we set 
\begin{eqnarray}\label{Vl}
V_l=\{P(\gamma^l) : \gamma(0)\in \Gamma_r, \gamma^l\cap (\Sigma_1\cup\Sigma_2)=\emptyset \} \ .
\end{eqnarray} 
\begin{lemma}\label{claim1a}
There exist $d_0,l_0>0$ such that if $d_{H}(\Sigma_1,\Sigma_2)\le d_0$ and $l\le l_0$ then 
\begin{eqnarray}
\Sigma_1 \cup\Sigma_2\subset \partial V_l \ .
\end{eqnarray}
\end{lemma}

This Lemma will be proved in Section \ref{properr}.

We shall use a variation of the Hausdorff distance which we call $l-$distance. 
\begin{definition}\label{dm}
We define the \emph{$l-$distance} $d_l$ between $\Sigma_1$ and $\Sigma_2$ as follows
\begin{eqnarray}
d_l(\Sigma_1,\Sigma_2)=\max \left\{\sup_{\substack {x\in \Sigma_1\cap \partial V_l}}\mbox{dist}(x,\Sigma_2)\ , \sup_{\substack {x\in \Sigma_2\cap \partial V_l}}\mbox{dist}(x,\Sigma_1)\  \right \}.
\end{eqnarray}

Here, $\sup_{\substack {x\in \Sigma_1\cap \partial V_l}}\mbox{dist}(x,\Sigma_2)$ is understood to be $0$ if $\Sigma_1\cap\partial V_{l}=\emptyset$ and analogously for $\sup_{\substack {x\in \Sigma_2\cap \partial V_l}}\mbox{dist}(x,\Sigma_1)$.

\end{definition}

\begin{lemma}\label{claim2}
There exist $d_1,l_1,C>0$ satisfying $0<d_1\le d_0, \ 0<l_1\le l_0$ such that, if $d_H(\Sigma_1,\Sigma_2)>d_1$ then 
\begin{eqnarray}
d_{l_1}(\Sigma_1,\Sigma_2)\ge C >0 \ .
\end{eqnarray}
\end{lemma}

See Section \ref{properr} for a proof of this Lemma. 

\begin{proposition}\label{equiv}
Let $\Omega, \Sigma_1,\Sigma_2$ be the domain and the cracks satisfying the a-priori assumptions stated above and let $l_1>0$ be the quantity introduced in Lemma \ref{claim1}.
Then, there exists a constant $C_1>0$ such that
\begin{eqnarray}
 d_{H}(\Sigma_1,\Sigma_2)\le C_1 d_{l_1}(\Sigma_1, \Sigma_2).
 \end{eqnarray}
\end{proposition}

This is an immediate consequence of the above two Lemmas, details can be found in Section \ref{properr}.



With no loss of generality, we can assume that there exists a point $O\in \Sigma_1\cap\partial V_{l_1}$ such that 
\begin{eqnarray}\label{puntodistanza}
d_{l_1}=d_{l_1}(\Sigma_1,\Sigma_2)=\mbox{dist}(O,\Sigma_2) \ .
\end{eqnarray}

\begin{proposition}\label{propagerrore}
Let $\Omega$ be the set in $\mathbb{R}^n$ satisfying the a-priori assumptions stated above. Let $l_1>0$ be the parameter introduced before and let $Q=P(\gamma^{\frac{l_1}{2}})\in \partial V_{\frac{l_1}{2}}$ be the vertex of the cone $RC_{\frac{l_1}{2}}(\gamma(1))$ for a given simple arc $\gamma$. 
Let $y=Q+h\widetilde{\nu}$, where $\widetilde{\nu}$ is the $RC_{\frac{l_1}{2}}(\gamma(1))$ cone axis unit vector.

If, given $\eps>0$, we have 
\begin{eqnarray}
\|\Lambda_1 -\Lambda_2 \|_{\mathcal{L}(H^{\frac{1}{2}}(\partial \Omega), H^{-\frac{1}{2}}(\partial \Omega))}\le \eps
\end{eqnarray}
then for every $0<h<\bar{h}$, we have that 
\begin{eqnarray}\label{stimaerrore}
|f(y,y)|\le C \frac{\eps^{C'h^F}}{h^{B}}
\end{eqnarray}
where $0<B<n-2$ and $\bar{h},C,C',F>0$ are constants depending on the a-priori data only. 
\end{proposition}
 \end{subsection}
 
Also, the proof of the above Proposition is postponed to Section \ref{properr}. 
 
 \begin{subsection}{Lower bound on the function f}

Let us consider $O\in \Sigma_1\cap \partial V_{l_1}$ the point in \eqref{puntodistanza}.
We introduce a point $O'\in \Sigma_1$ which is defined as follows by distinguishing two cases. 
\begin{itemize}
\item If $=O\in \Sigma_1$ is such that $\mbox{dist}(O,\partial\Sigma_1)<\frac{d_{l_1}}{4}$, then we consider a point $0'\in \Sigma_1$ so that $\mbox{dist}(O,O')=\frac{d_{l_1}}{2}$. It follows that $\mbox{dist}(O',\partial \Sigma_1)\ge \frac{d_{l_1}}{4}$ and $\mbox{dist}(O',\Sigma_2)\ge \frac{d_{l_1}}{2}$. 
\item If $O\in \Sigma_1$ is such that $\mbox{dist}(O,\partial\Sigma_1)\ge\frac{d_{l_1}}{4}$ then we set $O'=O$.
\end{itemize}

\begin{proposition}\label{stimealtobasso}
Let $\Omega$ be a domain in $\mathbb{R}^n$ satisfying the a-priori assumptions. Let $\Sigma_1, \Sigma_2$ be two cracks in $\Omega$ verifying the a-priori assumptions and $y=h\nu_1(O')$. Then for every $h, \ 0<h< \min\{\bar{r_0},r,\bar{h},\frac{c}{2}d_{l_1},\frac{c_0}{2}d_{l_1}^p\}$ we have that 
\begin{eqnarray}
|f(y,y)|\ge c_1 h^{2-n} - c_2|d_{l_1}^p-h|^{3-2n} 
\end{eqnarray}


where $c=\min\{\frac{1}{8},\frac{1}{4DnMr_0} \},p=\frac{n-1+\alpha}{\alpha}$ and $c_0,c_1,c_2>0$ are constants depending on the a-priori data only.
\end{proposition}
 
The proof is deferred to Section \ref{properr}.

 \end{subsection}

\begin{subsection}{Proof of the main Theorem}
We now give the proof of our main result. 


\begin{proof}[Proof of Theorem \ref{main}.]By Proposition \ref{stimeasintotiche} and Proposition \ref{stimealtobasso}, we have, up to a possible replacing of the constant $C$ in \eqref{stimaerrore},
that
\begin{eqnarray}
\frac{\eps^{C'h^F}}{h^{n-2}}\ge \frac{c_1}{C} h^{2-n} - \frac{c_4}{C}|d_{l_1}^p-h|^{3-2n} 
\end{eqnarray}
where $c_4=c_2+c_3$.
By choosing $h=qd_{l_1}^{p\frac{2n-3}{n-2}}$  we have that 
\begin{eqnarray}
\frac{c_1}{C} h^{2-n} - \frac{c_4}{C}|d_{l_1}^p-h|^{3-2n} \ge {c_5}h^{2-n}  \ .
\end{eqnarray}
with  $q=\frac{1}{8}\left(1+\left(\frac{c_1}{2c_4} \right)^{\frac{1}{2n-3}}\right)^{-\frac{2n-3}{n-2}}\left(\frac{c_1}{2c_4}\right)^{\frac{1}{n-2}}$ and $c_5=\min\{\frac{1}{2},\frac{c_1}{2C} \}$.
Hence by combining the last two inequalities we obtain 
\begin{eqnarray}
\eps^{C'h^F}> c_5\ ,
\end{eqnarray}
from which follows that 
\begin{eqnarray}
C'h^F\le \left|\frac{\log(c_5)}{\log(\eps)}\right|.
\end{eqnarray}
Finally, by our choice of $h$ we can conclude that 
\begin{eqnarray}
d_{l_1}\le c_6 |\log(\eps)|^{-\eta} \ ,
\end{eqnarray}
with $c_6=(q^{-F}C'^{-1}|\log(c_5)|)^{\frac{n-2}{Fp(2n-3)}}$ and $\eta=\frac{n-2}{Fp(2n-3)}$.

The thesis follows by Proposition \ref{equiv} with $C=c_6C_1$.
\end{proof}

\end{subsection}

\end{section}

\begin{section}{Proof of the regularity estimate}\label{regol}

In this section we shall give the proof of the regularity property of the solution $u$ to \eqref{C} and its first order derivatives near the crack.

\begin{proof}[Proof of Theorem \ref{rego}.]By the arguments in \cite[Chap. 3]{Eva}, $u$ is H\"{o}lder continuous with its first order derivatives up to $\Sigma$ except possibly at points of $\partial \Sigma$. The proof is based on the Moser iteration techniques (see for instance \cite[Chap. 8]{gt}) and by well known regularity bounds for the Neumann problem \cite[p. 667]{ADN}.

We now investigate the behavior of $u$ near the crack edge $\partial \Sigma$. Fix $x_0\in \partial \Sigma$, up to a translation we may assume that $x_0=0$. Let us consider the following change of variables $y=\hat{\Phi}(x)$

\begin{equation*}
\left\{
\begin{array}
{lcl}
y''=z''\ ,
\\
y_{n-1}= x_{n-1}- \psi(x'')\ ,
\\
y_n= x_n - \varphi(x'',x_{n-1}-\psi(x''))\ ,
\end{array}
\right.
\end{equation*}
where $\varphi, \psi$ are the $C^{1,\alpha}$ functions introduced in previous section satisfying \eqref{chap2:2}-\eqref{chap2:3} and \eqref{tip}-\eqref{psi2} respectively. The map $\hat{\Phi} \in C^{1,\alpha}(B_{\frac{r_0}{4M}}(0), \mathbb{R}^n)$ and its inverse $\hat{\Phi}^{-1}\in C^{1,\alpha}(B_{r_0}(0), \mathbb{R}^n)$.

With respect to the new variables the crack coincides, within $B_{\frac{r_0}{4M}}(0)$, with the half hyperplane $\{y_n=0, y_{n-1}<0 \}$.

Denoting with 
\begin{eqnarray}
&&\hat{A}(y)=|\mbox{det}D\hat{\Phi}^{-1}(y)|(D\hat{\Phi})(\hat{\Phi}^{-1}(y))(D\hat{\Phi})^T(\hat{\Phi}^{-1}(y))\ ,\\
&&\hat{\gamma}^+(y)=\gamma^+(\hat{\Phi}^{-1}(y))\ , \ \ \hat{\gamma}^-(y)=\gamma^-(\hat{\Phi}^{-1}(y)) \ , \\
&&v(y)=u(\hat{\Phi}^{-1}(y))\ 
\end{eqnarray}
we have that $v\in H^1(\hat{B}^-_{\frac{r_0}{4M}}(0) )$ is a weak solution to the problem 
\begin{equation}\
\left\{
\begin{array}
{lcl}
{\mbox{div}}(\hat{A}(y)\nabla v(y))=0\ ,& \mbox{in $\hat{B}_{\frac{r_0}{4M}}(0)$ ,}
\\
\hat{A}(y)\nabla v^{\pm}(y)\cdot \nu^{\pm} - \hat{\gamma}^{\pm}(y)v^{\pm}(y)=0 , & \mbox{on $\hat{B}^-_{\frac{r_0}{4M}}(0) $ ,}
\end{array}
\right.
\end{equation}
where $$\hat{B}_{\frac{r_0}{4M}}(0)=B_{\frac{r_0}{4M}}(0)\setminus \{y_n=0 \} $$ and $$\hat{B}^-_{\frac{r_0}{4M}}(0)=\hat{B}_{\frac{r_0}{4M}}(0)\cap \{y_{n-1}<0 \} $$ and $\nu^+=(0,\dots,0,-1), \ \nu^-=(0,\dots,0,1)$.

We introduce the following system of variables $z=\hat{\Psi}(y)$
\begin{equation*}
\left\{
\begin{array}
{lcl}
z''=y''\ , \ r=\sqrt{y_{n-1}^2 + y_n^2}
\\
z_{n-1}= \displaystyle\sqrt{\frac{r(r+y_{n-1})}{2}}\ ,
\\
z_n=sign(y_n) \displaystyle\sqrt{\frac{r(r-y_{n-1})}{2}}  ,
\end{array}
\right.
\end{equation*}

For the reader's convenience we express both systems of variables in cylindrical coordinates also
\begin{eqnarray}
y=(y'', r \cos \theta,r\sin \theta)
\end{eqnarray}

and
\begin{eqnarray}
z=(y'', r \cos \frac{\theta}{2},r\sin \frac{\theta}{2})
\end{eqnarray}

with $0<r<\frac{r_0\sqrt{n}}{8M}, -\pi\le\theta<\pi, -\frac{r_0\sqrt{n}}{8M}<h_i<\frac{r_0\sqrt{n}}{8M}, i=1\dots,n-2$.

The underlying idea here relies on mapping through $\hat{\Psi}$ the set $\hat{B}_{\frac{r_0\sqrt{n}}{8M}}(0)$ into the half ball $B^+_{\frac{r_0\sqrt{n}}{8M}}(0)=\{z\in {B}_{\frac{r_0\sqrt{n}}{8M}}(0) : z_{n-1}>0 \}$ so that the two side of the $\hat{B}^-_{\frac{r_0\sqrt{n}}{8M}}(0)$ are mapped into the flat part of $B^+_{\frac{r_0\sqrt{n}}{8M}}(0)$ (see also \cite[Remark C.3.1.]{ADiB}).

Moreover it can be verified that the map $\hat{\Psi}\in W^{1,\infty}(\hat{B}_{\frac{r_0\sqrt{n}}{8M}}(0), \mathbb{R}^n)$ and $\hat{\Psi}^{-1}\in W^{1,\infty}(B^+_{\frac{r_0\sqrt{n}}{8M}}(0), \mathbb{R}^n)$. 

Setting 
\begin{eqnarray*}
&&\hat{B}(y)=|\mbox{det}D\hat{\Psi}^{-1}(z)|(D\hat{\Psi})(\hat{\Psi}^{-1}(z))\hat{A}(\hat{\Psi}^{-1}(z))(D\hat{\Psi})^T(\hat{\Psi}^{-1}(z))\ ,\\
&&\bar{\gamma}(z)=
\left\{
\begin{array}{rl}
&\hat{\gamma}^+(\hat{\Psi}^{-1}(z))\ \ \mbox{if}\ z_n\ge 0 \ ,\\
& \hat{\gamma}^-(\hat{\Psi}^{-1}(z))\ \ \mbox{if}\ z_n< 0 \ ,
\end{array}
\right.\\
&&w(z)=v(\hat{\Psi}^{-1}(z))\ ,
 \end{eqnarray*}
we have that $w\in H^1(B^+_{\frac{r_0\sqrt{n}}{8M}}(O) )$ is a weak solution to the problem 
\begin{equation}\
\left\{
\begin{array}
{lcl}
{\mbox{div}}(\hat{B}(z)\nabla w(z))=0\ ,&& \mbox{in $B^+_{\frac{r_0\sqrt{n}}{8M}}(0)$ ,}
\\
\hat{B}(z)\nabla w(z)\cdot \hat{\nu} - \bar{\gamma}(z)w(z)=0 , && \mbox{on $B_{\frac{r_0\sqrt{n}}{8M}}(0)\cap \{z_{n-1}=0 \}$ ,}
\end{array}
\right.
\end{equation}
where $\hat{\nu}=(0,\dots,0,-1,0)$.

Observing that $\hat{B}\in L^{\infty}(B^+_{\frac{r_0\sqrt{n}}{8M}}(0)), \bar{\gamma}\in L^{\infty}(B_{\frac{r_0\sqrt{n}}{8M}}(0)\cap \{z_{n-1}=0 \})$ and dealing again as in \cite[Chap. 3]{Eva} we infer that $w\in C^{0,\alpha}(\overline{B^+_{\frac{r_0\sqrt{n}}{8M}}(0)})$.

Finally, coming back to the former system of coordinates, we obtain the thesis.  
\end{proof}
\end{section}
\begin{section}{Integration by parts, proofs}\label{sectionintbyparts}

In this Section we shall deal with the proof of our ``Alessandrini identity" type formula tuned for the crack problem at hand.

\begin{lemma}{(\bf{The divergence formula over $\Omega\setminus(\Gamma_1\cup\Gamma_2)$})}\label{divform}

Let $F$ be a vector field such that $F\in C^1(\Omega\setminus(\Gamma_1\cup\Gamma_2))$ and moreover $F\in C(\overline{\Omega^{i}_{+}})$ and $F\in C(\overline{\Omega^{i}_{-}})$ with $i=1,2$, then the following holds

\begin{eqnarray}\label{divergenceformula}
\int_{\Omega} {div}F dx =&& \int_{\Gamma_1\setminus\Gamma_2} [F\cdot\nu_1]_1 d\sigma(x) + \int_{\Gamma_2\setminus\Gamma_1} [F\cdot\nu_2]_2 d\sigma(x) + \nonumber \\ &&+ \int_{\Gamma_1\cap\Gamma_2} [F\cdot\nu_1]_1 d\sigma(x) +
 \int_{\partial \Omega} F\cdot {\nu}_{e}d\sigma(x) . 
\end{eqnarray}
\end{lemma}
\begin{proof}Given $0< \rho < r_0$, we have that by the compactness of $\overline{\Omega}$ we can find a finite number of points $P_i,\ i=1\dots,N $ such that $\cup_{i=1}^{N}B_{\rho}(P_i)$ covers $\overline{\Omega}$. Let us observe that due to the regularity hypothesis made on $\Gamma_i\ i=1,2$, we can choose $\rho$ small enough so that on each ball  $B_{\rho}(P_i)$ with $i=1,\dots, N$, $\Gamma_1$ and $\Gamma_2$ are separately graphs each with respect to a suitable reference system. 

Let $\{\alpha_i\}_{i=1}^{N}$ be a smooth partition of unity subordinate to the open covering $\cup_{i=1}^{N}B_{\rho}(P_i)$, namely we are assuming that 
\begin{description}
\item [i)] $0\le\alpha_i\le 1\ \ , \ \alpha_i \in C_0^{\infty}(B_{\rho}(P_i))\ , i=1,\dots, N$ ;

\item [ii)] $\sum_{i=1}^{N}\alpha_i =1$ on $\cup_{i=1}^{N}B_{\rho}(P_i)$.
 
\end{description}
Then, we have 
\begin{eqnarray}
\int_{\Omega} {\mbox{div}}F dx =\sum_{i=1}^N\int_{B_{\rho}(P_i)}\mbox{div} (\alpha_i F) dx \ .
\end{eqnarray}
The only interesting cases to consider are when $B_{\rho}(P_i)$ contains $\Gamma_1\cap\Gamma_2$, whereas in the other ones the divergence theorem may be applied in a straightforward fashion. Let us fix a small aperture $\theta_0$ and let us distinguish two cases. 
\begin{enumerate}
\item $\forall x_0\in B_{\rho}(P_i)\cap\Gamma_1\cap\Gamma_2$ the tangent planes of $\Gamma_1,\Gamma_2$ at $x_0$ have an aperture $\theta\ge\theta_0$.
\item $\exists x_0 \in B_{\rho}(P_i)\cap\Gamma_1\cap\Gamma_2$ and the tangent planes of $\Gamma_1, \Gamma_2$ in $x_0$ form an aperture $\theta<\theta_0$.
\end{enumerate}
{\bf{Case 1}}.
In such a case $ B_{\rho}(P_i)\setminus (\Gamma_1 \cup \Gamma_2)$ is composed by finitely many Lipschitz domains and the divergence theorem can be used in each component separately. Note that the same occurs when  $\Gamma_1\cap\Gamma_2\cap B_{\rho}(P_i)=\emptyset$.


{\bf{Case 2}}. 
In this situation, if one chooses $\theta_0$ and $\rho$ sufficiently small in terms of $r_0, M$ one obtain that there exists $x\in \Gamma_1\cap\Gamma_2\cap B_{\rho}(P_i)$ such that $\Gamma_1,\Gamma_2$ are tangential at $x$. In this case having chosen $\rho$ sufficiently small, $\Gamma_1$ and $\Gamma_2$ are simultaneously graphs with respect to the same reference system. Moreover, we consider the following three domains. 
\begin{description}
\item [a)] $U=\{(x',x_n)\in B_{\rho}(P_i):\ x_n> \max\{\varphi_1(x'),\varphi_2(x')\} \}$ ;
\item [b)] $I=\{(x',x_n)\in B_{\rho}(P_i): \min\{\varphi_1(x'),\varphi_2(x')\}<x_n< \max\{\varphi_1(x'),\varphi_2(x')\} \}$ ;
\item [c)] $L=\{(x',x_n)\in B_{\rho}(P_i):\ x_n< \min\{\varphi_1(x'),\varphi_2(x')\} \}$ .
\end{description}
Both $U$ and $L$ are Lipschitz domains. The set $I$ may not be Lipschitz and disconnected, but it is a normal domain between Lipschitz graphs. Hence in all such sets the divergence theorem holds true.   
\end{proof}

\begin{proof}[Proof of Theorem \ref{intbyparts}.] There exists a sequence of $C^{\infty}_0(\mathbb{R}^n)$ functions $\varphi_m, \ m\in \mathbb{N}$, satisfying the following properties. First, $0\le \varphi_m \le 1$, $\varphi_m$ is identically equal to $1$ in $\mathcal{U}^i_{\frac{1}{2m}}$ and $\varphi_{m}$ is identically equal to zero outside $\mathcal{U}_i^{\frac{1}{m}}$. Second, we have that 
\begin{eqnarray}
|\nabla \varphi_m|\le Cm\ \ \, \ \ |\mathcal{U}^i_{\frac{1}{m}}|\le \frac{C}{m^2}
\end{eqnarray}
and hence
\begin{eqnarray}\label{capacity}
\int_{\mathcal{U}^i_{\frac{1}{m}}}|\nabla\varphi_m(x)|^2\le C \ ,
\end{eqnarray}
where $C>0$ is a constant depending on the a-priori data only. 

We notice that 
\begin{eqnarray}
\mbox{div}(u_2\varphi_m\nabla u_1)= \nabla u_2\cdot \nabla u_1 \varphi_m + \nabla u_2\cdot\nabla \varphi_m u_1 . 
\end{eqnarray}
We observe that $\nabla u_1\cdot \nabla u_2 \varphi_m \rightarrow 0$ a.e. in $\Omega$ as $m\mapsto \infty$ and also in  $L^1(\Omega)$ by dominated convergence. 
On the other hand we have 
\begin{eqnarray}\label{capacity2}
\int_{\Omega}|\nabla u_2\cdot\nabla \varphi_m u_1|\le \left(\int_{\mathcal{U}^i_{\frac{1}{m}}}|\nabla\varphi_m(x)|^2\right)^{\frac{1}{2}}\left(\int_{\mathcal{U}^i_{\frac{1}{m}}}u_1^2|\nabla u_2(x)|^2\right)^{\frac{1}{2}} .
\end{eqnarray}
By the bound in \eqref{capacity} and observing that $u_1^2|\nabla u_2(x)|^2\in L^1(\Omega)$ we can conclude by the absolute continuity of the integral that the right hand side of \eqref{capacity2} tends to zero as $m\rightarrow \infty$. 

Hence we found that 
\begin{eqnarray}\label{approssimazione}
\int_{\Omega}\mbox{div}(u_2(x)\nabla u_1(x))dx = \lim_{m\rightarrow \infty}\int_{\Omega}\mbox{div}(u_2(x)(1-\varphi_m)\nabla u_1(x))dx
\end{eqnarray}
Using the divergence formula \eqref{divergenceformula} with $F=u_2(1-\varphi_m)\nabla u_1$ we get that

\begin{eqnarray}\label{divergenceformula2}
&&\int_{\Omega} \mbox{div}(u_2(1-\varphi_m)\nabla u_1) dx = \int_{\Gamma_1\setminus\Gamma_2} [u_2(1-\varphi_m)\partial_{\nu_1} u_1]_1 d\sigma(x) + \nonumber\\
&&+ \int_{\Gamma_2\setminus\Gamma_1} [u_2(1-\varphi_m)\partial_{\nu_2} u_1]_2 d\sigma(x) +  \int_{\Gamma_1\cap\Gamma_2} [u_2(1-\varphi_m)\partial_{\nu_1} u_1]_1 d\sigma(x) +\nonumber \\
 && \int_{\partial \Omega} u_2(1-\varphi_m)\partial_{{\nu}_{e}} u_1d\sigma(x) . 
\end{eqnarray}
Let $S$ be any of the portions $\Gamma_1\setminus\Gamma_2, \ \Gamma_2\setminus\Gamma_1, \ \Gamma_1\cap\Gamma_2$. We claim that 
\begin{eqnarray}\label{claim1}
\int_S \varphi_m[u_2\partial_{\nu_i} u_1]_i d\sigma(x) \rightarrow \ 0\ ,
\end{eqnarray}
with $i=2$ if $S=\Gamma_2\setminus\Gamma_1$ and with $i=1$ in the remaining cases. 

We observe that in order to prove our claim it sufficient to establish that $|\nabla u_1|\in L^1(S)$. If $\partial \Sigma_1 \cap S=\emptyset$ the integrability of $|\nabla u_1|$ over $S$ easily follows from Theorem \ref{rego}. 
Let us then analyze the case when $\partial \Sigma_1 \cap S \neq \emptyset$ and distinguish two situations.
\begin{description}
\item [i)] $\partial \Sigma_1$ intersects $S$ transversally;
\item [ii)] $\partial \Sigma_1$ intersects $S$ tangentially.
\end{description}
We begin by analyzing the case i) and observing that in such a case the intersection $V=\partial \Sigma_1\cap S$ is a $(n-3)-$manifold. We can find a finite number of points $P_i,\ i=1,\dots, N$ in $V$ such that $\cup_{i=1}^N B_{\hat{r}}(P_i)$ covers $V$, where $\hat{r}$ will be fixed later on. After a translation we may assume that $P_i=0$ and fixing local coordinates, we can represent $S$ as a graph of a $C^{1,\alpha}$ function  $\varphi$ satisfying \eqref{chap2:1}-\eqref{chap2:3}. Let $\Phi\in C^{1,\alpha}(B_{\frac{r_0}{4M}}(0), \mathbb{R}^n)$ be the map defined as follows 
\begin{eqnarray}
\Phi(y',y_n)=(y', y_n + \varphi(y'))\ .
\end{eqnarray} 
we have that there exists $\theta_1, \theta_2, \theta_1>1>\theta_2>0$ constants depending on $r_0$ and $M$ only such that for any $r\in (0,\frac{r_0}{4M})$ it follows 
\begin{eqnarray}
B_{\theta_2 r}(0) \subset \Phi(B_r(0))\subset B_{\theta_1 r}(0)\ .
\end{eqnarray}
The inverse map $\Phi^{-1} \in C^{1,\alpha}(B_{r_0}(0), \mathbb{R}^n)$ and it is defined by 
\begin{eqnarray}
\Phi^{-1}(x',x_n)=(x', x_n - \varphi(x'))\ .
\end{eqnarray}
Moreover, by our assumptions on $\partial \Sigma_1$ and by the implicit function theorem we have that there exists $\bar{r}>0$ depending on the a-priori data only such that for any $r\in (0,\bar{r})$
\begin{eqnarray*}
\Phi^{-1}(B_{\theta_2r}(0)\cap V)\subset \{y\in B'_r(0):  y_{n-1}=\psi_1(y_1, \dots, y_{n-2}), \ y_{n-2}=\psi_2(y_1, \dots, y_{n-3}) \} ,
\end{eqnarray*}
where $\psi_1 \in C^{1,\alpha}(B''_{r}(0), \mathbb{R})$ and $\psi_2 \in C^{1,\alpha}(B'''_{r}(0), \mathbb{R})$.
In particular, when $n=3$ the set $\Phi^{-1}(B_{\theta_2r(0)}\cap V)$ reduces to a single point. 

Let $\Psi \in C^{1, \alpha}(B'_{\frac{r}{4M}}, \mathbb{R}^{n-1})$ be the map defined as follows 
\begin{eqnarray}
\Psi(z''', z_{n-2}, z_{n-1})= (z''', z_{n-2}+\psi_2(z'''), z_{n-1} + \psi_1(z'')) \ .
\end{eqnarray}
As before it can be proved that there exist constants $\theta_3, \theta_4$ such that $\theta_3>1>\theta_4>0$ depending on $r_0$ and $M$ only such that 
for any $\rho \in (0,\frac{r}{4M})$ it follows that 
\begin{eqnarray}
B'_{\theta_4\rho}(0)\subset \Psi(B'_{\rho}(0))\subset B'_{\theta_3\rho}(0) \ .
\end{eqnarray}
The inverse map $\Psi^{-1}\in  C^{1, \alpha}(B'_{\rho}, \mathbb{R}^{n-1})$ and it is defined by 
\begin{eqnarray}
\Psi^{-1}(y''', y_{n-2}, y_{n-1})= (y''', y_{n-2}-\psi_2(y'''), y_{n-1} - \psi_1(y'')) \ .
\end{eqnarray}
Let $x$ be a point in $B_{\frac{\theta_4\bar{r}}{8M}}(0)$, then by Theorem \ref{rego}, we may infer that there exists a constant $C_1>0$ depending on the a-priori data only such that for any $\bar{x}\in V \cap B_{\frac{\theta_4\bar{r}}{8M}}(0)$ we have 
\begin{eqnarray}  
\frac{|u_1(x)-u_1(\bar{x})|}{|x-\bar{x}|}\le C |x-\bar{x}|^{\alpha-1}\le C_1 \mbox{dist}(x,  V \cap B_{\frac{\theta_4\bar{r}}{8M}}(O))^{\alpha -1} \ .
\end{eqnarray}
Let then $x$ be a point in $B_{\frac{\theta_4\bar{r}}{8M}}(0)\cap S$ and let $z'\in B'_{\frac{\bar{r}}{4M}}(0)$ be such that $x=\Phi(\Psi(z'),0)$. Furthermore, let $x_0\in V$ be such that $|x-x_0|={\mbox{dist}}(x, V)$ with $x_0=\Phi(\Psi(z''',0,0),0)$. We have that there exists constant $C_2>0$ depending on the a-priori data only such that 
\begin{eqnarray*}
|\nabla u_1(\Phi(\Psi(z'),0) )|\le C_2 {|\Phi(\Psi(z'),0) - \Phi(\Psi(z''',0,0),0) |}^{\alpha-1} \ .
\end{eqnarray*}
Finally, by the $C^{1,\alpha}$ regularity of $\Phi^{-1}$ and $\Psi^{-1}$ we can infer that there exists a constant $C_3>0$ depending on the a-priori data only such that 
\begin{eqnarray}
|\nabla u_1(\Phi(\Psi(z'),0) )|\le C_3 |(z',0) - (z''',0)|^{\alpha -1} \ .
\end{eqnarray}
From the above estimate we deduce that there exists a constant $C_4>0$ depending on the a-priori data only such that 
\begin{eqnarray}
\int_{B'_{\frac{\bar{r}}{4M}}(0)}\nabla u_1(\Phi(\Psi(z'),0) )|dz'\le C_4 \ .
\end{eqnarray}
Hence choosing $\hat{r}=\theta_2\theta_4\frac{\bar{r}}{16M}$ and by a covering argument, we obtain that 
\begin{eqnarray}
\int_{S} |\nabla u_1(x)|dx < \infty \ .
\end{eqnarray}
We now treat the case ii). Since in this case the intersection $\partial \Sigma_1\cap S$ might be an irregular set, we find convenient to consider the orthogonal projection operator $\Pi : \partial \Gamma_1 \rightarrow S$ and we define $W=\Pi(\partial \Gamma_1)$ which is an $(n-2)-$ manifold. As before, we can find a finite number of points $P_i, \ i=1,\dots, N$ in $W$ such that $\cup_{i=1}^N B_{\hat{r}}(P-i)$  covers $W$, where $\hat{r}$ will be chosen later on. Dealing as before we can locally flatten the hypersurface $S$ by the diffeomorphism $\Phi$.  

Furthermore, by our hypothesis on $\partial \Sigma_1$ we have that there exists $\widetilde{r}>0$ depending on the a-priori data only such that for any $r\in (0,\widetilde{r})$ we have 
\begin{eqnarray}
\Phi^{-1}(B_{\theta_2 r}(O)\cap W)\subset \{y'\in B'_r(0): y_{n-1}=\widetilde{\psi}(y_1, \dots, y_{n-2})\} \ ,
\end{eqnarray}
where $\widetilde{\psi}\in C^{1, \alpha}(B''_r(0), \mathbb{R})$.

Let $\widetilde{\Psi} \in C^{1, \alpha}(B'_{\frac{r}{4M}}, \mathbb{R}^{n-1})$ be the map defined as follows 
\begin{eqnarray}
\widetilde{\Psi}(z'', z_{n-1})= (z'', z_{n-1} + \widetilde{\psi}(z'')) \ .
\end{eqnarray}
There exist constants $\theta_5, \theta_6$ such that $\theta_5>1>\theta_6>0$ depending on $r_0$ and $M$ only such that 
for any $\rho \in (0,\frac{r}{4M})$ it follows that 
\begin{eqnarray}
B'_{\theta_6\rho}(0)\subset \widetilde{\Psi}(B'_{\rho}(0))\subset B'_{\theta_5\rho}(0) \ .
\end{eqnarray}
The inverse map $\widetilde{\Psi}^{-1}\in  C^{1, \alpha}(B'_{\rho}, \mathbb{R}^{n-1})$ and it is defined by 
\begin{eqnarray}
\widetilde{\Psi}^{-1}(y'', y_{n-1})= (y'', y_{n-1} - \widetilde{\psi}(y'')) \ .
\end{eqnarray}  

Let then $x$ be a point in $B_{\frac{\theta_4\bar{r}}{8M}}(0)\cap S$ and let $y\in B_{\frac{\theta_4\bar{r}}{8M}}(0)\cap \partial \Gamma_1$. Arguing as for case i) we can deduce by Theorem \ref{rego} that there exists a constant $C_5>0$ depending on the a-priori data only such that 
\begin{eqnarray}
|\nabla u_1(x)|\le&& C_5 |x-y|^{\alpha -1}\le C_5 |x- \Pi(y)|^{\alpha -1}\le\\ \nonumber
 \le&&C_5 \mbox{dist}(x, W\cap B_{\frac{\theta_4\bar{r}}{8M}}(0))^{\alpha -1}
\end{eqnarray}

Hence let $z'\in B_{\frac{\widetilde{r}}{4M}}(0)$ be such that $x=\Phi(\widetilde{\Psi}(z'),0)$ and let $y_0\in W$ be such that $\mbox{dist}(x,W\cap B_{\frac{\theta_4\bar{r}}{8M}}(0))= |x-y_0|$ with $y_0=\Phi(\widetilde{\Psi}(z'',0),0)$, then we have that there exists a positive constant $C_6>0$ depending on the a-priori data only such that 
\begin{eqnarray}
|\nabla u_1(\Phi(\widetilde{\Psi}(z'),0))|\le C_6 |(z',0)- (z'',0,0)|^{\alpha-1} \ .
\end{eqnarray}
 
Hence dealing as for the case i) and fixing the radius $\hat{r}=\theta_2\theta_6\frac{\widetilde{r}}{16M}$ we get that also in this situation $|\nabla u_1(x)|\in L^1(S)$.

Hence, our claim \eqref{claim1} is proved. Combining \eqref{approssimazione}, \eqref{divergenceformula2} and \eqref{claim1} we get that 
\begin{eqnarray}\label{integration1}
&&\int_{\partial\Omega}(\partial_{\nu_{e}}u_1u_2 - \partial_{\nu_{e}}u_2u_1 )  d\sigma= \int_{\Gamma_1\setminus\Gamma_2}(u_2[\partial_{\nu_1}u_1]_1 - [u_1]_1\partial_{\nu_1} u_2) d\sigma +\\
&& + \int_{\Gamma_2\setminus\Gamma_1}( [u_2]_2\partial_{\nu_2}u_1- u_1[\partial_{\nu_2}u_2]_2)d\sigma + \int_{\Gamma_1\cap\Gamma_2}([u_2\partial_{\nu_1}u_1]_1 - [u_1\partial_{\nu_2}u_2]_2) d\sigma . \nonumber
\end{eqnarray}
Noticing that the integrals over $\Gamma_i\setminus\Sigma_i, \ i=1,2$ cancel each other since $[u_i]_i$ and $[\partial u_i]_i$ vanish there, the formula \eqref{integration1} can be simplified as follows
\begin{eqnarray}
&&\int_{\partial\Omega}(\partial_{\nu_{e}}u_1u_2 - \partial_{\nu_{e}}u_2u_1)   d\sigma = \int_{\Sigma_1\setminus\Sigma_2}(u_2[\partial_{\nu_1}u_1]_1 - [u_1]_1\partial_{\nu_1} u_2) d\sigma +\\
&& + \int_{\Sigma_2\setminus\Sigma_1} ([u_2]_2\partial_{\nu_2}u_1- u_1[\partial_{\nu_2}u_2]_2)d\sigma + \int_{\Sigma_1\cap\Sigma_2}([u_2\partial_{\nu_1}u_1]_1 - [u_1\partial_{\nu_2}u_2]_2) d\sigma . \nonumber
\end{eqnarray} 
Finally, the desired identity follows by selfadjointness of the Dirichlet to Neumann map $\Lambda:H^{\frac{1}{2}}(\partial \Omega)\rightarrow H^{-\frac{1}{2}}(\partial \Omega)$. \end{proof}
\end{section}

\begin{section}{Proof of Proposition \ref{propagerrore}}\label{properr}
In this section we shall provide the proof of the Proposition \ref{propagerrore} together with the related auxiliary results stated in Section \ref{singularsolutions}.
 \begin{subsection}{The $l$-distance}
\begin{proof}[Proof of Lemma \ref{claim1a}.]
We give a sketch of the proof based on three steps. 
\\
\begin{enumerate}
\item Being $\Sigma_i, \ i=1,2$ contained into a $C^{1,\alpha}$ hypersurface $\Gamma_i,\ i=1,2$ and by the arguments carried over in \cite[Proposition 3.6]{ABRV} we may infer that there exist number $d_0,\rho_0, d_0>0, 0<\rho_0<r_0$ for which the ratio $\frac{d_0}{r_0},\frac{\rho_0}{r_0}$ only depend on $\alpha$ and $M$, such that if we have 
\begin{eqnarray}
d_H(\Sigma_1,\Sigma_2)\le d_0
\end{eqnarray}
then for any $P\in \overline{\Sigma_1}$ we have that 
\begin{eqnarray}
\Gamma_i\cap B_{\rho_0}(P)=\{ x\in B_{\rho_0}(P) \ \mbox{s.t.} \ x_n=\varphi_i(x') \} \ \ , \  i=1,2
\end{eqnarray}
and $\|\varphi_1-\varphi_2 \|_{C^{1,\frac{\alpha}{2}}(B_{\rho_0}(P))}\le C r_0^{\frac{2+\alpha}{2+2\alpha}}d_0^{\frac{\alpha}{2+2\alpha}}$, where $C>0$ depends on $\alpha$ and $M$ only. 

Moreover, the functions $\max\{\varphi_1,\varphi_2 \}$ and $\min\{\varphi_1,\varphi_2 \}$ are Lipschitz with Lipschitz constants bounded by $L=2M + C r_0^{\frac{2+\alpha}{2+2\alpha}}d_0^{\frac{\alpha}{2+2\alpha}}$.

\item We recall that in our regularity hypothesis for any $P\in \overline{\Sigma_i}, \  i=1,2$ we can define two unit normals $\nu_i(P)$ and $-\nu_i(P)$ with $i=1,2$ according with the criterion stated in Section \ref{directproblem}. Moreover, for any point $P\in \overline{\Sigma_i}, \ i=1,2$ we can find a set $\gamma_i^{\frac{r_0}{2}}$ with $A=\frac{1}{2M}$ so that $\Sigma_i\cap \gamma_i^{\frac{r_0}{2}}=\emptyset$ and the axis of the cone $RC^i_{\frac{r_0}{2}}$ (within $\gamma_i^{\frac{r_0}{2}}$) is oriented along $\nu_i(P), \ i=1,2$. The same holds true for the unit normal $-\nu_i(P)\ i=1,2$.

\item Let us assume that $d_{H}(\Sigma_1,\Sigma_2)\le d_0$ and let $P\in \Sigma_1\cup\Sigma_2$, say for instance $P\in \Sigma_1$. Without loss of generality, we may assume that in the local representation of $\Sigma_1$ and $\Sigma_2$ as relative graphs $P$ belongs to the graph of $\max\{\varphi_1,\varphi_2 \}.$ We notice that up to replacing $\frac{r_0}{2}$ with $l_0=\min \{\frac{\rho_0}{2},\frac{d_0}{2}\}$ and $A$ with $\frac{1}{2L}$ we have that the set $\gamma_1^{l_0}$ introduced in the previous step is such that $\gamma_1^{l_0}\cap \Sigma_2=\emptyset$ and this concludes the proof. 
\end{enumerate}
\end{proof}
\begin{proof}[Proof of Lemma \ref{claim2}.]Let $d_1=\min\{\frac{d_0}{2},\frac{r_0}{4}\}$. If $d_{H}(\Sigma_1,\Sigma_2)\ge d_1$ then without loss of generality we may infer that there exists $Q\in \Sigma_2\cap ({E_1}_{d_1})$.\\
 We can find a continuous path $\gamma:[0,1] \rightarrow\ {E_1}_{d_1}$ such that $\gamma(0)\in \Gamma_r$ and $\gamma(1)=Q$. Let now $Q'=\gamma(\bar{t})$ with $\bar{t}=\inf\{t : \gamma(t) \not\in {\Sigma_2}_{l_1} \}$ and let $P'\in \Sigma_2$ such that $|P'-Q'|\le l_1$ with $l_1=\min\{\frac{l_0}{2},\frac{d_1}{2}\}$. Furthermore, let us denote with $\gamma'$ the restriction to $[0,\bar{t}]$ of $\gamma$. Up to a possible replacement of the constant $l_1$ we have that $P'\in \partial V_{l_1}$ (where the construction of such a $V_{l_1}$ is based on the path $\gamma'$ and $A=\frac{1}{2M}$). 
Finally we get 
\begin{eqnarray}
\mbox{dist}(P',\Sigma_1)\ge d_1 -l_1>0
\end{eqnarray} 
which implies that 
\begin{eqnarray}
d_{l_1}(\Sigma_1,\Sigma_2)\ge d_1 -l_1>0.
\end{eqnarray} 
The thesis follows with $C=d_1 -l_1$. 
\end{proof}

\begin{proof}[Proof of Proposition \ref{equiv}.]We distinguish two cases 
 \begin{enumerate}
 \item $d_{H}(\Sigma_1,\Sigma_2)\le d_0$,
 \item $d_{H}(\Sigma_1,\Sigma_2)>d_0$,
 \end{enumerate}
 where $d_0$ is the constant introduced in Lemma \ref{claim1}.
 
 {\bf{Case 1}}. In such a case we have by Lemma \ref{claim1} that $\Sigma_i \cap \partial V_{l_1}=\Sigma_i, \ i=1,2$ and hence $d_{H}(\Sigma_1,\Sigma_2)=d_{l_1}(\Sigma_1,\Sigma_2)$.
 
 {\bf{Case 2}}. In this situation we have that being $d_{H}(\Sigma_1,\Sigma_2)>d_1$, by Lemma \ref{claim2} we can infer that  $d_{l_1}(\Sigma_1,\Sigma_2)>C$. Hence we get 
 \begin{eqnarray}
 d_{H}(\Sigma_1, \Sigma_2)\le \frac{d_{H}(\Sigma_1,\Sigma_2)}{C}d_{l_1}(\Sigma_1,\Sigma_2)\le \frac{D}{C}d_{l_1}(\Sigma_1, \Sigma_2).
 \end{eqnarray}
 where $D$ is the a-priori bound on the diameter of $\Omega$ introduced in \eqref{diametro}.
 The thesis follows by choosing $C_1=\max\left\{1,\frac{D}{C}\right\}$.
 \end{proof}

\end{subsection}


\begin{subsection}{Proof of the upper bound on $f$}

\begin{proof}[Proof of Proposition \ref{propagerrore}]Let $\gamma$ be the simple arc in the definition of $V_{\frac{l_1}{2}}$ and let $x_1=\gamma(0)\in \Gamma_r$ with $0<r<\frac{l_1}{16}$.
Let us fix $\bar{y}\in \mathcal{S}_{\frac{r}{4}, 4r}$, where  $\mathcal{S}_{\frac{r}{4}, 4r}$ is the set introduced before.
Let us consider $f(\bar{y},\cdot)$ and let $\gamma$ be the simple arc in the definition of $V_{\frac{l_1}{2}}$, then we have that 
\begin{eqnarray}
\Delta_w f(\bar{y},w)=0 \ \ \mbox{in}\ \ \Omega_{\Sigma}^c 
\end{eqnarray}
 For $w\in \mathcal{S}_{\frac{r}{4}, 4r}$, by \eqref{errore}, \eqref{eguaglianzafond} and \eqref{asin1}
\begin{eqnarray}\label{stimaerrore2}
 |f(\bar{y},w)|\le C \|\Lambda_1 -\Lambda_2 \|= C\eps \ .
\end{eqnarray}
By the three spheres inequality for supremum norms of harmonic function we have that there exists a constant $0<\tau<1$ such that  
\begin{eqnarray}
\|f(\bar{y},\cdot)\|_{L^{\infty}(B_{\frac{3r}{2}}(x_1))}\le \|f(\bar{y},\cdot)\|^{\tau}_{L^{\infty}(B_{\frac{r}{2}}(x_1))}\|f(\bar{y},\cdot)\|^{1-\tau}_{L^{\infty}(B_{{2r}}(x_1))} \ .
\end{eqnarray}
We consider a point $\bar{w}$ lying on the arc $\gamma$ and such that $\bar{w}\in \gamma^{\frac{l_1}{2}}\setminus B_{2h}(Q)$. Let us define $\{x_i\}, \ i=1,\dots,s$ as follows, $x_1$ it has been already introduced, $x_{i+1}=\gamma(t_i)$ where $t_i=\max\{t. |\gamma(t)-x_i|=r \}$ if $|x_i-\bar{w}|>r$, otherwise let $i=s$ and stop the process. By construction, the balls $B_{\frac{r}{2}}(x_i)$ are pairwise disjoint, $|x_{i+1}-x_i|=r$ for $i=1,\dots,s-1$, $|x_s-\bar{w}|\le r$. By \eqref{diametro} we have that there exists a positive constant $\beta$ depending on the a priori data only such that $s<\beta$.
An iterated use of the three spheres inequality for $f(\bar{y},\cdot)$ gives that for any $0<\rho<r$ we have 
\begin{eqnarray}
\|f(\bar{y},\cdot)\|_{L^{\infty}(B_{\frac{\rho}{2}}(w))}\le \|f(\bar{y},\cdot)\|^{\tau^s}_{L^{\infty}(B_{\frac{\rho}{2}}(x_1))}\|f(\bar{y},\cdot)\|^{1-\tau^s}_{L^{\infty}(\gamma^{\frac{l_1}{2}}\setminus B_{2h}(Q))} \ .
\end{eqnarray}
We observe that for $w\in \gamma^{\frac{l_1}{2}}\setminus B_{2h}(Q)$ we have that 
\begin{eqnarray}
&&|S_{\Sigma_1}(\bar{y},w)|\le\nonumber\\
&\le& \int_{\Sigma_1\setminus\Sigma_2}(|R_2(x,w)[\partial_{\nu_1}R_1(x,\bar{y})]_1| +|[R_1(x,\bar{y})]_1\partial_{\nu_1} R_2(x,w) |)d\sigma +\nonumber\\
&+& \int_{\Sigma_1\cap\Sigma_2}|[R_2(x,w)\partial_{\nu_1}R_1(x,\bar{y})]_1|  d\sigma(x) \le\nonumber\\
&\le& C\left(\int_{\Sigma_1\setminus\Sigma_2}(|x-w|^{2-n} +|x-w|^{1-n} d\sigma(x) + \int_{\Sigma_1\cap\Sigma_2}|x-w|^{2-n})  d\sigma(x)\right)\le\nonumber\\
&\le& C h^{-1}\ .
\end{eqnarray}
Similarly, we get $|S_{\Sigma_2}(\bar{y},w)|\le C h^{-1}.$ Then we can conclude that 
\begin{eqnarray}\label{primastimaf}
|f(\bar{y},w)|\le C h^{-1} \ \ \mbox{for any} \ w\in \gamma^{\frac{l_1}{2}}\setminus B_{2h}(Q).
\end{eqnarray}
Hence, we have that by \eqref{stimaerrore2} and \eqref{primastimaf} 
\begin{eqnarray}
\|f(\bar{y},\cdot)\|_{L^{\infty}(B_{\frac{\rho}{2}}(w))}\le C \eps^{\tau^s}h^{\tau^s-1}.
\end{eqnarray}

We introduce the following set of quantities for $k\ge 2$
\begin{eqnarray}
&&\theta=\arctan{\left(\frac{1}{2A}\right)},\ \ \theta_1=\arctan\left(\frac{1}{4A} \right),\\
&&\chi=\frac{1-\sin (\theta_1)}{1+\sin(\theta_1)},\\
&&\lambda_1=\min\left\{\frac{r_0}{1+\sin(\theta)}, \frac{r_0}{3\sin(\theta)}\right\},\ \ \lambda_k=\chi\lambda_{k-1}\\
&&\rho_1=\lambda_1\sin(\theta_1)\ , \ \rho_k=\chi\rho_{k-1}\\
&&w_1=Q+\lambda_1\widetilde{\nu} \ , \ w_k=Q+\lambda_k\widetilde{\nu}.
\end{eqnarray}
By repeating the argument outlined in \cite[Proposition 3.5]{ADC} (see also \cite{ABRV}) and based on iterative application of the three spheres inequality over a chain of balls $B_{\rho_1}(w_1), \dots, B_{\rho_{k(r)}}(w_{k(r)})$ within the cone  we obtain that
\begin{eqnarray}\label{primopasso}
\|f(\bar{y},\cdot)\|_{L^{\infty}(B_{\rho_{k(r)}}(w_{k(r)}))}\le C \eps^{\tau^s\tau^{k(r)}-1}h^{\tau^s-1} \ ,
\end{eqnarray} 
where $k(r)$ is an integer such that $k(r)\sim \frac{\left|\log(\frac{r}{Al_1})\right|}{|\log(\chi)|}$ with $0<r<\min\{\frac{Al_1}{64}, r_0 \}$.

Let us now consider $f(y,w)$ as a function of $y$. First, we observe that 
\begin{eqnarray}
\Delta_y f(y,w)=0 \ \ \mbox{in}\ \ \Omega_{\Sigma}^c, \ \ \mbox{for any}\ \ w\in \Omega_{\Sigma}^c.
\end{eqnarray}
For $y,w \in \gamma^{\frac{l_1}{2}}\setminus B_{h}(Q), y\neq w$ we have that 
\begin{eqnarray}
|S_{\Sigma_1}({y},w)|&&\le\int_{\Sigma_1\setminus\Sigma_2}|R_2(x,w)[\partial_{\nu_1}R_1(x,{y})]_1| +|[R_1(x,{y})]_1\partial_{\nu_1}R_2(x,w) |d\sigma +\nonumber\\
&&+ \int_{\Sigma_1\cap\Sigma_2}|[R_2(x,w)\partial_{\nu_1}R_1(x,{y})]_1|  d\sigma(x) \le\nonumber\\
&&\le C\int_{\Sigma_1\setminus\Sigma_2}|x-w|^{2-n}|x-y|^{2-n} +|x-y|^{2-n}|x-w|^{1-n} d\sigma(x) + \nonumber\\
&& + C\int_{\Sigma_1\cap\Sigma_2}|x-w|^{2-n} |x-y|^{2-n} d\sigma(x) \nonumber
\end{eqnarray}
Moreover, dealing as in Proposition \ref{stimeasintotiche} we get 
\begin{eqnarray}
&&|S_{\Sigma_1}({y},w)|\le 
\left\{
\begin{array}{rl}
&h^{-1}|\log|y-w||\ \ \mbox{if}\ n=3 \ ,\\
& h^{-1}|y-w|^{3-n}\ \ \ \ \mbox{if}\ n>3\ .
\end{array}
\right.
\end{eqnarray}
and similarly for $S_{\Sigma_2}({y},w)$. Therefore, 
\begin{eqnarray}\label{secondopasso}
|f(y,w)|\le \left\{
\begin{array}{rl}
&h^{-1}|\log|y-w||\ \ \mbox{if}\ n=3 \ ,\\
& h^{-1}|y-w|^{3-n}\ \ \ \ \mbox{if}\ n>3\ .
\end{array}
\right.
\end{eqnarray}
with $y,w \in \gamma^{\frac{l_1}{2}}\setminus B_{h}(Q)$.
Moreover, for $y\in \mathcal{S}_{\frac{r}{4}, 4r}$ and $w\in \gamma^{\frac{l_1}{2}}\setminus B_{h}(Q)$ using \eqref{primopasso} we have 
\begin{eqnarray}\label{propag1}
|f(y,w)|\le C  \eps^{\tau^s\tau^{k(h)}-1}h^{\tau^s-1} \ .
\end{eqnarray}
Proceeding as before, let us fix $w\in\gamma^{\frac{l_1}{2}}$ such that $\mbox{dist}(w,Q)=h$ and $\widetilde{y}\in \mathcal{S}_{\frac{r}{4}, 4r}$. 
Again, taking $y_1=Q+\lambda_1\widetilde{\nu}$ and using iteratively the three spheres inequality we have 
\begin{eqnarray}\label{tres}
\|f(\cdot,w) \|_{L^{\infty}(B_{\frac{r_0}{2}}(y_1))}\le \|f(\cdot,w) \|^{\tau^s}_{L^{\infty}(B_{\frac{r_0}{2}}(\widetilde{y}))}\|f(\cdot,w) \|^{1-\tau^s}_{L^{\infty}(\gamma^{\frac{l_1}{2}}\setminus B_{2h}(Q))}
\end{eqnarray}
where $\tau$ and $s$ are the numbers established previously. 
We now distinguish two cases
\begin{description}
\item [i)] $n=3$
\item [ii)] $n>3$
\end{description}
We begin by analyzing the case $i)$.

By combining \eqref{secondopasso}, \eqref{propag1} and \eqref{tres} we have 
\begin{eqnarray}
\|f(\cdot,w) \|_{L^{\infty}(B_{\frac{r_0}{2}}(y_1))}\le C h^{\tau^{2s}-1}|\log h|^{1-\tau^s}\eps^{\tau^{2s}\tau^{k(h)-1}} \ .
\end{eqnarray}
We observe that for $h$ sufficiently small we have that $|\log h|^{1-\tau^s}\le h^{-\frac{1}{2}\tau^{2s}}$. And hence from the above estimate we deduce that 
\begin{eqnarray}
\|f(\cdot,w) \|_{L^{\infty}(B_{\frac{r_0}{2}}(y_1))}\le C h^{\frac{1}{2}\tau^{2s}-1}\eps^{\tau^{2s}\tau^{k(h)-1}} \ .
\end{eqnarray}

Once more, we apply iteratively the three spheres inequality over a chain of balls contained in the cone $RC_{\frac{l_1}{2}}(\gamma(1))$ and we obtain
\begin{eqnarray}
\|f(y,w) \|_{L^{\infty}(B_{\rho_{k(h)}}(y_k(h)))}\le C h^{( \frac{1}{2}\tau^{2s}-1)(1-\tau^{k(h)-1})}\eps^{\tau^{2s}\tau^{2(k(h)-1)}}.
\end{eqnarray}
From the above inequality, choosing $y=w=Q+2h\widetilde{\nu}$ we have that 
\begin{eqnarray}
|f(y,y)|\le C h^{-B}\eps^{\tau^{2s}\tau^{2(k(h)-1)}}
\end{eqnarray}
where $B=1-\frac{1}{2}\tau^{2\beta}$. We observe that, for $0<h<cr_0$ with $0<c<1$ depending on the a-priori data only, we have $k(h)\le c|\log h|=-c\log h$, so we deduce that $$\tau^{2k(h)}\ge \exp(-2c\log h\log \tau)= h^F$$ with $F=2c|\log \tau|$.

Finally we obtain that 
\begin{eqnarray}\label{end}
|f(y,y)|\le h^{-B}\eps^{\tau^{2\beta}\tau^{(2k(h)-2)}}\le C  h^{-B}\eps^{\tau^{2\beta-2}h^F}.
\end{eqnarray}
Hence the thesis follows with $\bar{h}=cr_0, C'=\tau^{2\beta-2}$.

For the case $ii)$ the estimate \eqref{end} holds true with $B=n-2 - \tau^{\beta}(n-3)-\tau^{2\beta}$ the other constants remaining the same and can be achieved by adapting the argument above.
\end{proof}
\end{subsection}
\end{section}

\begin{section}{Proof of Proposition \ref{stimealtobasso}}\label{stimeab}
We premise the proof of Proposition \ref{stimealtobasso} with several preliminary results. 
\begin{lemma}\label{soluzionepositiva}
There exists a constant $C>0$ depending on the a-priori data only, such that the weak solution $v\in H^1(\Omega\setminus \Sigma)$ to the problem 
\begin{equation}\label{Positiva}
\left\{
\begin{array}
{lcl}
\Delta v=0\ ,& \mbox{in $\Omega\setminus\overline{\Sigma}$ ,}
\\
 \partial_{\nu}v= 1\ ,   & \mbox{on $\partial\Omega$ ,}
\\
\partial_{\nu^{\pm}}v^{\pm} -\gamma^{\pm} v^{\pm} =0 \ , & \mbox{on $\Sigma^{\pm}$ .}
\end{array}
\right.
\end{equation}
is such that $v(x)\ge C$ in $\overline{\Omega}$.
\end{lemma}
\begin{proof}
The existence and the uniqueness of the weak solution  $v\in H^1(\Omega\setminus \Sigma)$ to the problem \eqref{Positiva} is a consequence of standard theory on the boundary value problem for the Laplace equation and the non negativity of the coefficients $\gamma^+$ and $\gamma^-$.
We understand that $v$ satisfies 
\begin{eqnarray}\label{weakpos}
\int_{\Omega\setminus\overline{\Sigma}}\nabla v\cdot\nabla \varphi \ dx + \int_{\Sigma}\gamma^+v^+\varphi^+ d\sigma + \int_{\Sigma}\gamma^-v^-\varphi^- d\sigma = 
\int_{\partial\Omega}\varphi \ d\sigma \ \
\end{eqnarray}
for any $\varphi\in H^1(\Omega\setminus \Sigma)$.

Let $v_- \in H^1(\Omega\setminus \Sigma)$ be the negative part of $v$, namely $v_-= - \min \{v, 0 \}$. Choosing $\varphi=v_-$ in \eqref{weakpos} we have that 
\begin{eqnarray}
\int_{\Omega\setminus\overline{\Sigma} \cap \{v\le0 \}}|\nabla v_-|^2 \ dx + \int_{\Sigma \cap \{v\le0 \}}\gamma^+|v^+_-|^2 d\sigma + \int_{\Sigma \cap \{v\le0 \}}\gamma^-|v^-_-|^2 d\sigma\le 0 \ .
\end{eqnarray}
Then by the Poincar\'{e} inequality we deduce that $v_- \equiv 0$ a.e. in $\overline{\Omega}$ and hence $v(x)\ge 0$ a.e. in $\overline{\Omega}$.

Let $x_0\in \partial (\Omega\setminus \Sigma)$ be such that $\displaystyle{\min_{x\in \overline{\Omega}} v(x)}= v(x_0)$. 

As a consequence of the Giraud's maximum principle (see \cite[Theorem 5]{Giraud}) we have that $x_0\in \overline{\Sigma}$. Without loss of generality, we may assume that if $x_0\in \Sigma$ then $\displaystyle{\min_{x\in \overline{\Omega}} v(x)}= v^+(x_0)$. 

Let us denote for any $0<\rho<r_0$ with 
$$
\Delta_{\rho}(x_0)=
\left\{
\begin{array}{rl}
&B_{\rho}(x_0)\cap \Omega^+ \ \ \mbox{if}\ \ x_0\in \Sigma\\
&B_{\rho}(x_0)\setminus \overline{\Sigma}\ \ \ \ \mbox{if}\ \ x_0\in \partial\Sigma .
\end{array}
\right.
$$ 
By the weak Harnack inequality at the boundary (see \cite[Lemma 3.2]{Si}) and by the non negativity of $v$ we have that there exist a radius $\widetilde{r}, \ 0<\widetilde{r}< r_0$ and a constant $C>0$ depending on the a-priori data only such that for any $0<\rho<\widetilde{r}$ we have 
\begin{eqnarray}
v(x_0)=\displaystyle{\min_{x\in \Delta_{\rho}(x_0)} v(x)}\ge C \|v\|_{L^2(\Delta_{2\rho}(x_0))} \ .
\end{eqnarray}
Moreover, dealing as in the proof of Lemma 3.3 of \cite{Si} and relying on an iterated use of the Harnack inequality we can conclude that there exists a constant $C>0$ depending on the a priori data only such that $v(x_0)\ge C$.
\end{proof}

We now introduce the following notion. 
Let $\gamma_1>0$ be a constant. We shall refer to $R_{\Omega}$ as the following Robin function 
\begin{equation}\label{Ro}
\left\{
\begin{array}
{lcl}
\Delta_x R_{\Omega}(x,y)=- \delta(x-y)\ ,&& \mbox{in $\Omega\setminus\overline{\Sigma}$ ,}
\\
\partial_{\nu}R_{\Omega}(\cdot, y)+ \gamma_1 R_{\Omega}(\cdot, y)=0\ ,   && \mbox{on $\partial\Omega$ ,}
\\
\partial_{\nu^{\pm}}R_{\Omega}^{\pm}(\cdot, y) - \gamma^{\pm}(\cdot) R_{\Omega}^{\pm}(\cdot, y) =0 \ , && \mbox{on $\Sigma^{\pm}$ ,}
\end{array}
\right.
\end{equation}
with $y\in \Omega\setminus \Sigma.$

\begin{lemma}\label{boundR}
Let $R$ be the solution to \eqref{R} and let $0<r<r_0, \ y\in \Omega\setminus \overline{\Sigma}$ be such that $B_{2r}(y)\subset \Omega\setminus\overline{\Sigma}$ and $\mbox{dist}(y,\partial \Omega)> r_0$. Then there exists a constant $c_r>0$ depending on the a-priori data and on $r$ only such that 
\begin{eqnarray}
\|R(\cdot,y)\|_{L^{\infty}(\Omega\setminus B_r(y))}\le c_r \ .
\end{eqnarray} 
\end{lemma}

\begin{proof}
Let $f\in L^{\frac{n+1}{2}}(\Omega)$ and let $u\in H^1(\Omega\setminus\overline{\Sigma})$ be the weak solution to 

\begin{equation}\label{u}
\left\{
\begin{array}
{lcl}
\Delta u= f \ ,& \mbox{in $\Omega\setminus\overline{\Sigma}$ ,}
\\
\partial_{\nu}u + \gamma_1 u=0\ ,   & \mbox{on $\partial\Omega$ ,}
\\
\partial_{\nu^{\pm}}u^{\pm} - \gamma^{\pm} u^{\pm}=0 \ , & \mbox{on $\Sigma^{\pm}$ .}
\end{array}
\right.
\end{equation}
By Green's second formula the solution $u$ can be represented as follows 
\begin{eqnarray}\label{repr}
u(y)= - \int_{\Omega}R_{\Omega}(x,y)f(x)dx \  
\end{eqnarray}
where $y\in \Omega\setminus \overline{\Sigma}$.
By the argument in \cite[Section 8.5]{gt} it follows that there exists a constant $C>0$ depending on the a-priori data only such that 
\begin{eqnarray}\label{globbound}
\|u\|_{L^{\infty}(\Omega\setminus \overline{\Sigma})}\le C \left(\|u\|_{L^2(\Omega\setminus\overline{\Sigma})} + \|f\|_{L^{\frac{n+1}{2}}(\Omega\setminus\overline{\Sigma})} \right) .
\end{eqnarray}
Moreover combining the weak formulation of problem \eqref{u}, the Poincar\'{e} and the H\"{o}lder inequalities we have that there exists a constant $C>0$ depending on the a-priori data only such that 
\begin{eqnarray}
\|u\|^2_{L^2(\Omega\setminus\overline{\Sigma})}\le C \|f\|_{L^{\frac{n+1}{2}}(\Omega\setminus\overline{\Sigma})}\cdot\|u\|_{L^{\frac{n+1}{n-1}}(\Omega\setminus\overline{\Sigma})} \ .
\end{eqnarray} 
Furthermore, being $1<\frac{n+1}{n-1}\le 2$ we may infer that 
\begin{eqnarray}
\|u\|_{L^2(\Omega\setminus\overline{\Sigma})}\le  C \|f\|_{L^{\frac{n+1}{2}}({\Omega\setminus\overline{\Sigma}})} \ .
\end{eqnarray} 
where $C>0$ is a constant depending on the a-priori data only. 
Hence inserting the above estimate in \eqref{globbound} we get that 
\begin{eqnarray}\label{stimainf}
\|u\|_{L^{\infty}(\Omega\setminus \overline{\Sigma})}\le C \|f\|_{L^{\frac{n+1}{2}}(\Omega\setminus\overline{\Sigma})} \ .
\end{eqnarray}
where $C>0$ is a constant depending on the a-priori data only. 

Hence \eqref{repr} and \eqref{stimainf} yield to
\begin{eqnarray}\label{repr2}
\|R_{\Omega}(\cdot,y)\|_{L^{\frac{n+1}{n-1}}(\Omega\setminus\overline{\Sigma})}&=& \sup_{\displaystyle{\|f\|_{L^{\frac{n+1}{2}}(\Omega\setminus\overline{\Sigma})}=1}}\left|\int_{\Omega}f(x)R_{\Omega}(x,y)dx \right| \ .\nonumber
\end{eqnarray}
Finally by the weak Harnack inequality (see \cite[Section 8.6]{gt}), we have that there exists a constant $C_r$ depending on the a-priori data and on $r$ only such that 
\begin{eqnarray}\label{Har}
\|R_{\Omega}(\cdot, y)\|_{L^{\infty}(\Omega\setminus B_r(y))}\le C_r \|R(_{\Omega}\cdot, y) \|_{L^{\frac{n+1}{n-1}}(\Omega\setminus\overline{\Sigma})} \ .
\end{eqnarray} 

Combining \eqref{repr2} and \eqref{Har} we obtain that 
\begin{eqnarray}\label{star}
\|R_{\Omega}(\cdot, y)\|_{L^{\infty}(\Omega\setminus B_r(y))}\le \widetilde{C_r}
\end{eqnarray}
where $\widetilde{C_r}>0$ is a constant depending on the a-priori data only. 
Finally let us now consider the harmonic function $u_{\Omega}(\cdot)=R(\cdot, y)-R_{\Omega}(\cdot, y)$ in $H^1(\Omega\setminus\overline{\Sigma})$. It follows that $u_{\Omega}$ solves 

\begin{equation}
\left\{
\begin{array}
{lcl}
\Delta u_{\Omega}= 0 \ ,& \mbox{in $\Omega\setminus\overline{\Sigma}$ ,}
\\
\partial_{\nu}u_{\Omega}=\partial_{\nu}R(\cdot,y) +\gamma_1 R_{\Omega}(\cdot, y) \ ,   & \mbox{on $\partial\Omega$ ,}
\\
\partial_{\nu^{\pm}}u_{\Omega}^{\pm} - \gamma^{\pm} u_{\Omega}^{\pm}=0 \ , & \mbox{on $\Sigma^{\pm}$ .}
\end{array}
\right.
\end{equation}
By \eqref{star} and by standard asymptotic estimate on the gradient of $R(\cdot,y)$ we get
\begin{eqnarray}
\|\partial_{\nu}u_{\Omega}\|_{L^{\infty}(\partial \Omega)}\le c r_0^{1-n} + \gamma_0  \widetilde{C_r} \ .
\end{eqnarray} 
Classical estimates for harmonic functions leads to the existence of a constant $C_r>0$ depending on the a-priori data only such that 
\begin{eqnarray}\label{dstar}
\|u_{\Omega}\|_{L^{\infty}(\Omega\setminus \overline{\Sigma})}\le C_r \ .
\end{eqnarray}
Hence combining \eqref{star} and \eqref{dstar} we obtain the thesis.
\end{proof}

We now introduce the following notion.

Let $\gamma_0\ge 0$ be a constant. We shall denote with $R_0$ the half space Robin function 
\begin{equation}\label{R0}
\left\{
\begin{array}
{lcl}
\Delta_x R_0(x,y)= -\delta(x-y)\ ,&& \mbox{in $\mathbb{R}^n\setminus \{x_n=0 \}$ ,}
\\
\partial_{\nu^{\pm}}R_0^{\pm}(\cdot, y) - \gamma_0 R_0^{\pm}(\cdot, y) =0 \ , && \mbox{on $\{x_n=0 \}$ ,}
\end{array}
\right.
\end{equation}
with $y_n\neq 0$ and $\nu^+=(0,\dots,0,-1)$ and $\nu^-=(0,\dots,0,1)$.

\begin{proposition}\label{stimeasintotiche}
Let $\Sigma$ be a crack satisfying the a-priori assumption stated above. Let $\rho>0$ and let $x\in \Sigma^{\rho}$. Then there exists a constant $c_1,c_2,c_3>0$ depending on the a-priori data only such that 
\begin{description}
\item [i)]\begin{eqnarray}\label{asin1}
|\nabla_z R(z,y)|\le {c_0} |z-y|^{1-n}\ \ , \ \ \ \ \ \ \ \ \ \ \ \ \ \  \ \ \ \ \ \ \ \ \ \ \ 
\end{eqnarray}
for any $y,z\in \mathbb{R}^n$.
\item [ii)] \begin{eqnarray}\label{asin2}
|R(z,y)-R_0(z,y)|\le \frac{c_1}{\bar{r_0}^{\alpha}} |z-y|^{2-n+\alpha}\ \ , \ \ \ \ \ \ \ \ \ \ 
\end{eqnarray}
\begin{eqnarray}\label{asin3}
|\nabla_z R(z,y)-\nabla_z R_0(z,y)|\le \frac{c_2}{\bar{r_0}^{{\alpha}^2}} |z-y|^{1-n+{\alpha}^2}\ \ , 
\end{eqnarray}
for any $z\in \Sigma \cap B_r(x)$ and for any $y=h\nu(x)$ with $0<r<\bar{r_0}, \ 0<h<\bar{r_0}$ where $\bar{r_0}=c_3\min\{r_0, \rho \}$ and $\gamma_0$ in \eqref{R0} is such that $\gamma_0=\gamma^+(x)$.
\end{description}
\end{proposition}

\begin{proof}
Without loss of generality we may assume that $x=0$. 
Let $\rho_0=\frac{1}{4}\min\{r_0, \mbox{dist}(x,\partial \Sigma) \}$ and let $\Phi \in C^{1,\alpha}(B_{\frac{\rho_0}{4M}}, \mathbb{R}^n)$ be the map introduced in Theorem \ref{intbyparts}. In particular we have that for any $0<r<{\frac{\rho_0}{4M}}$ it follows 
\begin{eqnarray}
\Omega_+\cap B_{\theta_2 r}(0)\subset \Phi(B_r^-(0))\subset \Omega_+\cap B_{\theta_1 r}(0)
\end{eqnarray}
where $\theta_1$ and $\theta_2$ are the constants mentioned in Theorem \ref{intbyparts}. 

We divide the proof in two steps. 

{\bf{i)}} In the first step we shall prove that there exists a constant $C_1>0$ depending on the a-priori data only such that 
\begin{eqnarray}\label{primastima}
|\nabla_z R(z,y)|\le C_1 |z-y|^{1-n}\ \ \ \mbox{for every}\ \ z,y \in \overline{\Omega_+\cap B_{\theta_2 {\frac{\rho_0}{8M}} }(0)} \ ,
\end{eqnarray} 
the other cases being trivial. 
Let then $z,y \in \overline{\Omega_+\cap B_{\theta_2 {\frac{\rho_0}{8M}} }(0)}$ and let $\zeta=\Phi^{-1}(z), \eta=\Phi^{-1}(y) \in \overline{B^-_{{\frac{\rho_0}{8M}}}(0)}$.

Denoting by 
\begin{eqnarray}
A(\zeta)=|\mbox{det}D\Phi(\zeta)|(D\Phi^{-1})(\Phi(\zeta))(D\Phi^{-1})^T(\Phi(\zeta))\ , 
\end{eqnarray}
\begin{eqnarray}
\widetilde{\gamma}^+(\zeta)=\gamma^+(\Phi(\zeta))\ , \ \widetilde{\gamma}^-(\zeta)=\gamma^-(\Phi(\zeta))
\end{eqnarray}
\begin{eqnarray}
\widetilde{R}(\zeta,\eta)=R(\Phi(\zeta),\Phi(\eta))
\end{eqnarray}
it follows that 
\begin{eqnarray}
C_1|\xi|^2\le A(\zeta)\xi\cdot\xi\le C_2|\xi|^2	 \ , \ \forall\ \zeta \in \overline{B^-_{{\frac{\rho_0}{8M}}}(0)}, \ \forall \xi \in \mathbb{R}^n \ ,
\end{eqnarray}
\begin{eqnarray}
|A(\zeta_1)- A(\zeta_2)|\le C_3 |\zeta_1-\zeta_2|^{\alpha}\ ,\ \forall \ \zeta_1, \zeta_2 \in \overline{B^-_{{\frac{\rho_0}{8M}}}(0)}
\end{eqnarray}
where $C_1,C_2, C_3>0$ are constants depending on the a-priori data only. Let us observe that $\widetilde{R}(\zeta,\eta)$ satisfies

\begin{equation}\
\left\{
\begin{array}
{lcl}
-{\mbox{div}}_{\zeta}(A(\zeta)\nabla_{\zeta}\widetilde{R}(\zeta,\eta))=\delta(\zeta-\eta)\ ,& \mbox{in $B^-_{{\frac{\rho_0}{8M}}}(0) $ ,}
\\
A(\zeta)\nabla_{\zeta}\widetilde{R}(\zeta,\eta)\cdot \nu' + \widetilde{\gamma}^+(\zeta)\widetilde{R}(\zeta,\eta)=0 , & \mbox{on $B'_{{\frac{\rho_0}{8M}}}(0) $ ,}
\end{array}
\right.
\end{equation}
where $\nu'=(0,\dots,0,1)$. 

Let $v\in H^1(\Omega\setminus\overline{\Sigma})$ be the solution to the problem \eqref{Positiva} and let $\widetilde{v}(\zeta)=v(\Phi(\zeta))$. Since by Lemma \ref{soluzionepositiva} we have that $\widetilde{v}(\zeta)\ge \bar{C}$ in $B^-_{\frac{\rho_0}{8M}}(0)$ then the quotient 
\begin{eqnarray}\label{quoziente}
\widetilde{N}(\zeta, \eta)= \frac{\widetilde{R}(\zeta, \eta)}{\widetilde{v}(\zeta)}
\end{eqnarray} 
is well defined there. 

Moreover, observing that $A(\zeta)=(a_{i,j}(\zeta))\ , i,j=1,\dots,n$ is a symmetric matrix, we have that straightforward calculations lead to 

\begin{equation}\
\left\{
\begin{array}
{lcl}
-{\mbox{div}}_{\zeta}(B(\zeta)\nabla_{\zeta}\widetilde{N}(\zeta,\eta))=\delta(\zeta-\eta)\ ,& \mbox{in $B^-_{{\frac{\rho_0}{8M}}}(0) $ ,}
\\
B(\zeta)\nabla_{\zeta}\widetilde{N}(\zeta,\eta)\cdot \nu' =0 , & \mbox{on $B'_{{\frac{\rho_0}{8M}}}(0) $ ,}
\end{array}
\right.
\end{equation}
where $B(\zeta)=(b_{i,j}(\zeta))= (\widetilde{v}^2(\zeta)a_{i,j}(\zeta))$.

Writing for any $\zeta, \eta \in B^-_{{\frac{\rho_0}{8M}}}(0)$

$$
\widetilde{N}_e(\zeta, \eta)=
\left\{
\begin{array}{rl}
&\widetilde{N}(\zeta, \eta) \ \ \mbox{if}\ \ \zeta=(\zeta',\zeta_n) \ \mbox{is such that} \ \zeta_n\le 0 \\
&\widetilde{N}(\zeta^*, \eta) \ \ \mbox{if}\ \ \zeta=(\zeta',\zeta_n)\  \mbox{is such that} \ \zeta_n> 0   ,
\end{array}
\right.
$$ 
and for $i=j$ with $i,j=1,\dots,n$
$$
{b}_{i,j}^e(\zeta)=
\left\{
\begin{array}{rl}
&{b_{i,j}}(\zeta) \ \ \mbox{if}\ \ \zeta=(\zeta',\zeta_n) \ \mbox{is such that} \ \zeta_n\le 0 \\
&{b_{i,j}}(\zeta^*) \ \ \mbox{if}\ \ \zeta=(\zeta',\zeta_n)\  \mbox{is such that} \ \zeta_n> 0   ,
\end{array}
\right.
$$ 
whereas for $i\neq j, i,j=1 \dots, n$ we set
$$
{b}_{i,j}^e(\zeta)=
\left\{
\begin{array}{rl}
&{b_{i,j}}(\zeta) \ \ \mbox{if}\ \ \zeta=(\zeta',\zeta_n) \ \mbox{is such that} \ \zeta_n\le 0 \\
&-{b_{i,j}}(\zeta^*) \ \ \mbox{if}\ \ \zeta=(\zeta',\zeta_n)\  \mbox{is such that} \ \zeta_n> 0   ,
\end{array}
\right.
$$ 
where $\zeta^*=(\zeta',-\zeta_n)$. The first two are even and the third one is odd with respect to $\{\zeta_n=0 \}$. In particular we have 
\begin{eqnarray}
-{\mbox{div}}_{\zeta}(B^e(\zeta)\nabla_{\zeta}\widetilde{N}_e(\zeta,\eta))=\delta(\zeta-\eta)+\delta(\zeta-\eta^*) \ \mbox{in} \ \ B_{\frac{\rho_0}{8M}}(0) 
\end{eqnarray}

where $B^e(\zeta)=(b_{i,j}^e(\zeta))$.

Write 
\begin{eqnarray}\label{decomp}
\widetilde{N}_e(\zeta,\eta)= \hat{N}_e(\zeta,\eta) + \hat{N}_e(\zeta,\eta^*)
\end{eqnarray}
where 
\begin{eqnarray} 
-{\mbox{div}}_{\zeta}(B^e(\zeta)\nabla_{\zeta}\hat{N}_e(\zeta,\eta))=\delta(\zeta-\eta) \ \ \mbox{in} \ \ B_{\frac{\rho_0}{8M}}(0) . 
\end{eqnarray}
Let $\eta\in  B_{\frac{\rho_0}{16M}}(0)$. By Lemma \ref{boundR} we have that there exists a constant $C>0$ depending on the a-priori data only such that 
\begin{eqnarray}\label{coronacirc}
\|\hat{N}_e(\cdot,\eta)\|_{L^{\infty}\left(B_{\frac{15\rho_0}{128}}(0)\setminus B_{\frac{13\rho_0}{128}}(0)\right)}\le C \ . 
\end{eqnarray}
We now consider the Green function $G(\zeta, \eta)$ such that 

\begin{equation}\
\left\{
\begin{array}
{lcl}
-{\mbox{div}}_{\zeta}(B^e(\zeta)\nabla_{\zeta}G(\zeta,\eta))=\delta(\zeta-\eta)\ ,& \mbox{in $B_{{\frac{7\rho_0}{8M}}}(0) $ ,}
\\
G(\zeta,\eta)=0 , & \mbox{on $\partial B_{{\frac{7\rho_0}{8M}}}(0) $ ,}
\end{array}
\right.
\end{equation}
with $\eta\in  B_{\frac{\rho_0}{16M}}(0)$. By the pointwise bound of $G$ with the fundamental solution for the Laplace equation (see \cite{LSW}) we infer that 
\begin{eqnarray}
|G(\zeta, \eta)|\le C |\zeta - \eta|^{2-n} \ \ \ \ \forall \ \zeta \in B_{{\frac{7\rho_0}{8M}}}(0)\ ,\ \ \forall \ \eta \in B_{\frac{\rho_0}{16M}}(0) \ 
\end{eqnarray}
where $C>0$ is a constant depending on the a-priori data only. 

Let us define $w(\zeta, \eta)=\hat{N}_e(\zeta,\eta)- G(\zeta,\eta)$, then we have 

\begin{equation}\
\left\{
\begin{array}
{lcl}
{\mbox{div}}_{\zeta}(B^e(\zeta)\nabla_{\zeta}w(\zeta,\eta))=0\ ,& \mbox{in $B_{{\frac{7\rho_0}{8M}}}(0) $ ,}
\\
w(\zeta,\eta)=\hat{N}_e(\zeta,\eta) , & \mbox{on $\partial B_{{\frac{7\rho_0}{8M}}}(0) $ .}
\end{array}
\right.
\end{equation}
Then by the bound in \eqref{coronacirc} and the maximum principle for solutions to equations in divergence form we have that 
\begin{eqnarray}
\|w(\cdot,\eta)\|_{L^{\infty}(B_{{\frac{7\rho_0}{8M}}}(0) )} \le C 
\end{eqnarray}
 where $C>0$ is a constant depending on the a-priori data only. 
 
Hence we may infer that there exists a constant $C>0$ depending on the a-priori data only such that 
\begin{eqnarray}
 |\hat{N}_e(\zeta, \eta)|\le C |\zeta - \eta|^{2-n} \ \ \ \ \forall \zeta, \forall \ \eta \in B_{\frac{\rho_0}{16M}}(0) \ ,\  \xi\neq \eta.
\end{eqnarray}
Moreover recalling \eqref{decomp} we have that 
\begin{eqnarray}\label{asinN}
 |\widetilde{N}_e(\zeta, \eta)|\le C |\zeta - \eta|^{2-n} \ \ \ \ \forall \zeta, \forall \ \eta \in B_{\frac{\rho_0}{16M}}(0) \ ,\  \xi\neq \eta \ ,
\end{eqnarray}
where $C>0$ is a constant depending on the a-priori data only. 

We observe that by Theorem \ref{rego} we have that the function $\widetilde{v}(\zeta)\ge C$ in $B^-_{\frac{\rho_0}{8M}}(0)$ where $C>0$ is a constant depending on the a-priori data only. Thus, by \eqref{quoziente} and by \eqref{asinN} we get that 

\begin{eqnarray}\label{asinR}
|\widetilde{R}(\zeta, \eta)|\le C |\zeta - \eta|^{2-n}  \ \ \forall \zeta, \forall \ \eta \in \overline{B^-_{\frac{\rho_0}{16M}}(0)} \ , \ \xi\neq \eta \ ,
\end{eqnarray}
where $C>0$ is a constant depending on the a-priori data only.

Let $h=\mbox{dist}(0,\eta)=|\eta|$, then we have 

\begin{equation}\
\left\{
\begin{array}
{lcl}
{\mbox{div}}_{\zeta}(A(\zeta)\nabla_{\zeta}\widetilde{R}(\zeta,\eta))=0\ ,& \mbox{in $B^-_{{\frac{h}{2}}}(0) $ ,}
\\
A(\zeta)\nabla_{\zeta}\widetilde{R}(\zeta,\eta)\cdot \nu' + \widetilde{\gamma}^+(\zeta)\widetilde{R}(\zeta,\eta)=0 , & \mbox{on $B'_{{\frac{h}{2}}}(0) $ ,}
\end{array}
\right.
\end{equation}
By well-known regularity bounds for the Neumann problem (see for instance \cite[p.667]{ADN}) we have that 
\begin{eqnarray}
\|\nabla_{\zeta} \widetilde{R}(\cdot,\eta)\|_{L^{\infty}({B^-_{\frac{h}{4}}(0)})}\le \frac{C}{h}\|\widetilde{R}(\cdot,\eta)\|_{L^{\infty}({B^-_{\frac{h}{2}}(0)})} \ ,
\end{eqnarray}
where $C>0$ is a constant depending on the a-priori data only.
By Theorem (regularity) we claim that there exists $\bar{\zeta}\in \overline{B^-_{\frac{h}{2}}(0)}$ such that 
\begin{eqnarray}
|\widetilde{R}(\bar{\zeta},\eta)|= \|\widetilde{R}(\cdot,\eta)\|_{L^{\infty}(\overline{B^-_{\frac{h}{2}}(0)})} .
\end{eqnarray}
Then by \eqref{asinR} we find that 
\begin{eqnarray}
\|\nabla_{\zeta} \widetilde{R}(\cdot,\eta)\|_{L^{\infty}(\overline{B^-_{\frac{h}{4}}(0)})}\le\frac{C}{h} |\bar{\zeta}-\eta|^{2-n} \ , 
\end{eqnarray}
where $C>0$ is a constant depending on the a-priori data only.

On the other hand, noticing that 
\begin{eqnarray}
|\bar{\zeta}-\eta|\ge |\eta| - |\bar{\zeta}|\ge h - \frac{h}{2}= \frac{h}{2}
\end{eqnarray}
we obtain that 

\begin{eqnarray}
\|\nabla_{\zeta} \widetilde{R}(\cdot,\eta)\|_{L^{\infty}(\overline{B^-_{\frac{h}{4}}(O)})}\le{C} |\eta|^{1-n} \ , 
\end{eqnarray}
where $C>0$ is a constant depending on the a-priori data only.

Coming back to the original coordinates we have 

\begin{eqnarray}
|\nabla_z R(z,y)|=|D\Phi^{-1}(z)^{T}\nabla_{\zeta}\widetilde{R}(\Phi^{-1}(z),\Phi^{-1}(y))|\le C_1 |z-y|^{1-n} \ , 
\end{eqnarray}
where $C_1>0$ is a constant depending on the a-priori data only.

{\bf{ii)}} In the second step we shall achieve the desired asymptotic estimates. 

Arguing as in \cite[Proposition 3.4]{ADC} we consider a function $\theta \in C^{\infty}(\mathbb{R})$ such that $0\le \theta \le 1, \ \theta(t)=1$, for $|t|<1$, $\theta(t)=0$, for $|t|>2$ and $\left|\frac{d\theta}{dt}\right|\le 2$. Let us fix $\rho_1=\min\{\frac{1}{4}(8M)^{-\frac{1}{\alpha}},\frac{1}{4} \}\cdot \frac{\theta_2\rho_0}{8M}$ and let us consider the following change of variables $z=\bar{\Phi}(\zeta)$ defined by

\begin{equation*}
\left\{
\begin{array}
{lcl}
\zeta'=z'
\\
\zeta_n= z_n - \varphi(z')\theta\left(\frac{|z'|}{\rho_1}\right)\theta\left(\frac{z_n}{\rho_1}\right)\ .
\end{array}
\right.
\end{equation*}
It can be verified that the map $\bar{\Phi}$ is a $C^{1, \alpha}(\mathbb{R}^n, \mathbb{R}^n)$ which satisfies the following properties 
\begin{eqnarray}\label{1}
\bar{\Phi}(Q^-_{\rho_1}(0))=\Omega_+ \cap Q_{\rho_1}(0)
\end{eqnarray} 
\begin{eqnarray}\label{2}
c^{-1}|z_1-z_2|\le |\bar{\Phi}^{-1}(z_1)-\bar{\Phi}^{-1}(z_2)|\le c |z_1-z_2| \ , \ \forall \ z_1,z_2 \in \mathbb{R}^n \ ,
\end{eqnarray}
\begin{eqnarray}\label{3}
|\bar{\Phi}^{-1}(z)- z|\le \frac{c}{{\rho_0}^{\alpha}} |z|^{1+\alpha} \ , \ \forall \ z \in \mathbb{R}^n \ ,
\end{eqnarray}
\begin{eqnarray}\label{4}
|D\bar{\Phi}^{-1}(z)- I|\le \frac{c}{{\rho_0}^{\alpha}} |z|^{\alpha} \ , \ \forall \ z \in \mathbb{R}^n \ ,
\end{eqnarray}
where $Q^-_{\rho_1}(0)= \{\zeta \in Q_{\rho_1}(O): \zeta_n<0 \}$ being $Q_{\rho_1}(0)$ the cube centered in $O$ with sides of length $2\rho_1$ and parallel to the coordinated axes and where $c>0$ is a constant depending on $M$ and $\alpha$ only.  

Let us define the half cylinder $C^-_{\rho_1}$ as 
\begin{eqnarray}
C_{\rho_1}= \{z \in \mathbb{R}^n : |z'|< \rho_1 , \ - \rho_1< z_n <0  \} \ .
\end{eqnarray}
For $z,y \in C^-_{\rho_1}$, we have that $\bar{R}(\zeta, \eta)= R(\bar{\Phi}(\zeta), \bar{\Phi}(\eta))$ is a solution to

\begin{equation}\
\left\{
\begin{array}
{lcl}
-{\mbox{div}}_{\zeta}(\bar{A}(\zeta)\nabla_{\zeta}\bar{R}(\zeta,\eta))=\delta(\zeta-\eta)\ ,& \mbox{in $ C^-_{\rho_1}$ ,}
\\
\bar{A}(\zeta)\nabla_{\zeta}\bar{R}(\zeta,\eta)\cdot \nu' + \overline{\gamma^+}(\zeta)\bar{R}(\zeta,\eta)=0 , & \mbox{on $B'_{\rho_1}(0)$ ,}
\end{array}
\right.
\end{equation}
where $z= \bar{\Phi}(\zeta), y= \bar{\Phi}(\eta), \overline{\gamma^+}(\zeta)= \gamma^+(\bar{\Phi}(\zeta))$ and where $$\bar{A}(\zeta)=|\mbox{det}D\bar{\Phi}(\zeta)|(D\bar{\Phi}^{-1})(\bar{\Phi}(\zeta))(D\bar{\Phi}^{-1})^T(\bar{\Phi}(\zeta)).$$ Moreover, we observe that $\bar{R}$ is of class $C^{\alpha}$ and $\bar{A}(0)=I$. Let $R_0(\zeta,\eta)$ be the fundamental solution introduced in \eqref{R0} with $\gamma_0=\overline{\gamma^+}(0)$. We notice that there exists a constant $C>0$ depending on the a-priori data only such that $|\overline{\gamma^+}(\zeta')-\gamma_0|\le C |\zeta'|^{\alpha}$ for any $ \zeta'\in B'_{\rho_1}(0)$.

Let us consider 
\begin{eqnarray}
\bar{M}(\zeta,\eta)= \bar{R}(\zeta,\eta)- R_0(\zeta,\eta)
\end{eqnarray}
which satisfies 

\begin{equation*}\
\left\{
\begin{array}
{lcl}
\Delta_{\zeta}\bar{M}(\zeta,\eta))={\mbox{div}}_{\zeta}((I-\bar{A})(\zeta)\nabla_{\zeta}\bar{R}(\zeta,\eta))\ ,&& \mbox{in $ C^-_{\rho_1}$ ,}
\\
\nabla_{\zeta}\bar{M}(\zeta,\eta)\cdot \nu' + {\gamma_0}(\zeta)\bar{M}(\zeta,\eta)=\\
=(I-\bar{A})\nabla\bar{R}(\zeta,\eta)\cdot \nu' + (\gamma_0-\overline{\gamma^+}(\zeta))\bar{R}(\zeta,\eta)\ , && \mbox{on $B'_{\rho_1}(0)$ .}
\end{array}
\right.
\end{equation*}
Let $L>0$ be such that $\overline{\Omega}\subset B_L(0)$. Thus by the representation formula over $B_L^-(0)$ we get 
\begin{eqnarray*}
\bar{M}(\zeta, \eta)&=&\int_{C^-_{\rho_1}}(I-\bar{A})(\xi)\nabla_{\xi}\bar{R}(\xi,\zeta)\nabla_{\xi}{R_0}(\xi,\eta)d\xi +\\ &+&\int_{B'_{\rho_1}(0)}(\gamma_0-\overline{\gamma^+}(\xi'))\bar{R}((\xi',0),\zeta)R_0((\xi',0),\eta)d\xi'+  \\
 &+&\int_{B_L^-(0)\setminus C^-_{\rho_1} }(I-\bar{A})(\xi)\nabla_{\xi}\bar{R}(\xi,\zeta)\nabla_{\xi}{R_0}(\xi,\eta)d\xi + \\
 &+& \int_{\partial(B_L^-(0))\setminus B'_{\rho_1}(0)} (\bar{A}-I)(\xi)\nabla_{\xi}\bar{R}(\xi,\zeta)\cdot\nu{R_0}(\xi,\eta) +\\
 &+& \int_{\partial(B_L^-(0))\setminus B'_{\rho_1}(0)} \partial_{\nu}\bar{M}(\xi,\zeta)R_0(\xi,\eta) - \partial_{\nu}R_0(\xi,\eta)\bar{M}(\zeta,\xi)d \sigma(\xi)\ . 
\end{eqnarray*}
For $|\zeta|,|\eta|\le \frac{\rho_1}{2}$ the last two integrals are bounded. Moreover, by \eqref{primastima} we have that
\begin{eqnarray*}
|\bar{M}(\zeta,\eta)|\le&& C \left (1 + \int_{C^-_{\rho_1}} |\xi|^{\alpha}|\xi-\zeta|^{1-n}|\xi-\eta|^{1-n}d\xi \right)+\\
&+&C\left(\int_{B'_{\rho_1}(0)} |\xi'|^{\alpha}|(\xi',0)-\zeta|^{2-n}|(\xi',0)-\eta|^{2-n}d\xi' \right) =\\
&=& C (1+ I_1 + I_2 + I_3 + I_4) \ ,
\end{eqnarray*}
where $C$ depends on the a-priori data only and 
\begin{eqnarray*}
&&I_1 = \int_{C^-_{\rho_1}\cap \{|\xi|<4h \}} |\xi|^{\alpha}|\xi-\zeta|^{1-n}|\xi-\eta|^{1-n}d\xi \ , \\
&&I_2 = \int_{C^-_{\rho_1}\cap \{|\xi|>4h \}} |\xi|^{\alpha}|\xi-\zeta|^{1-n}|\xi-\eta|^{1-n}d\xi \ , \\
&&I_3=\int_{B'_{\rho_1}(0)\cap \{|\xi'|<4h \}}|\xi'|^{\alpha}|(\xi',0)-\zeta|^{2-n}|(\xi',0)-\eta|^{2-n}d\xi' \ , \\
&&I_4=\int_{B'_{\rho_1}(0)\cap \{|\xi'|>4h \}}|\xi'|^{\alpha}|(\xi',0)-\zeta|^{2-n}|(\xi',0)-\eta|^{2-n}d\xi' \ , \\
\end{eqnarray*}
with $h=|\zeta-\eta|$.

We bound $I_1$ as follows 
\begin{eqnarray}
I_1&\le& h^{\alpha +2 -n}\int_{|\hat{\xi}|<4}|\hat{\xi}|^{\alpha}|\hat{\zeta}-\hat{\xi}|^{1-n}|\hat{\eta}-\hat{\xi}|^{1-n}d\hat{\xi} \le \nonumber\\
&\le& 4^{\alpha}h^{\alpha +2 -n}\int_{|\hat{\xi}|<4}|\hat{\zeta}-\hat{\xi}|^{1-n}|\hat{\eta}-\hat{\xi}|^{1-n}d\hat{\xi} \ , 
\end{eqnarray}
where $\hat{\xi}= \frac{\xi}{h}, \hat{\zeta}= \frac{\zeta}{h}, \hat{\eta}= \frac{\eta}{h}$. 
From standard bounds (see for instance \cite[Chapter 2]{Mi}) it follows that 
\begin{eqnarray*}
\int_{|\hat{\xi}|<4}|\hat{\zeta}-\hat{\xi}|^{1-n}|\hat{\eta}-\hat{\xi}|^{1-n}d\hat{\xi}<\infty
\end{eqnarray*}
for any $\hat{\zeta}, \hat{\eta}\in \mathbb{R}^n, \ |\hat{\zeta}-\hat{\eta}|=1$. Thus we found that 
\begin{eqnarray*}
I_1 \le c |\zeta-\eta|^{\alpha +2 -n} \ .
\end{eqnarray*}
Let us now consider $I_2$. We recall that by our hypothesis we have that $|\eta|=-\eta_n$. Let $\zeta=(\zeta',\zeta_n)$ be such that $|\zeta_n|<\frac{1}{4}|\eta_n|$. Then we have that $h=|\zeta-\eta|\ge \frac{1}{2}|\eta|$ from which we deduce that $|\xi|\le 2|\xi-\eta|$ and $|\xi|\le 4|\xi-\zeta|$. Hence we obtain that 
\begin{eqnarray*}
I_2\le c \int_{\{|\xi|>4h \}} |\xi|^{\alpha+2-2n}d\xi\le c h^{\alpha+2-n} \ .
\end{eqnarray*}
Treating analogously the integrals $I_3$ and $I_4$ we find that 
\begin{eqnarray}\label{Stimaasint}
|\bar{M}(\zeta,\eta)|\le C |\zeta-\eta|^{\alpha +2 -n}\ , 
\end{eqnarray}
for any $\eta=(0,\cdots,0, \eta_n)$ such that $0<-\eta_n< \frac{\rho_1}{2}$ and for any $|\zeta|\le \frac{\rho_1}{2} $ such that $|\zeta_n|<\frac{1}{4}|\eta_n|$ and where $C>0$ is a constant depending on the a-priori data only.  

Furthermore by Theorem \ref{rego} we have $\bar{M}(\cdot,\eta)\in C^{\alpha}\left({\overline{C^-_{\frac{\rho_1}{2},\frac{1}{4}\eta_n}}}\right)$ 
where $C^-_{\frac{\rho_1}{2},\frac{1}{4}\eta_n}=\{\zeta\in C^-_{\rho_1}:|\zeta|\le \frac{\rho_1}{2}, |\zeta_n|\le \frac{1}{4}\eta_n \} $. Hence we can deduce that the above estimate remains valid for points $|\zeta|\le \frac{\rho_1}{2}$ such that $\zeta =(\zeta',0)$.

We now go back to the original coordinates system. 

Let $z\in \bar{\Phi}(B'_{\frac{\rho_1}{2}}(0))$ and let $y=(0,y_n)$ with $y_n\in (-\frac{\rho_1}{2},0)$, then since $\bar{\Phi}^{-1}(y)=y$ and since $|\bar{\Phi}^{-1}(y)|=|\bar{\Phi}^{-1}(y)-\bar{\Phi}^{-1}(0)|\le |\bar{\Phi}^{-1}(y)-\bar{\Phi}^{-1}(z)|$ we get by \eqref{2} that 
\begin{eqnarray}\label{5}
c^{-1}|z|\le |\bar{\Phi}^{-1}(z)|\le |\bar{\Phi}^{-1}(y)-\bar{\Phi}^{-1}(z)|+|\bar{\Phi}^{-1}(y)|\le c|y-z|.
\end{eqnarray}
On the other hand by \eqref{3} and by \eqref{5} we have that 
\begin{eqnarray}\label{6}
|\bar{\Phi}^{-1}(z)-z|\le \frac{c}{\rho_0^{\alpha}}|z|^{1+\alpha}\le \frac{c'}{\rho_0^{\alpha}} |z-y|^{1+\alpha} \ .
\end{eqnarray}
We have that 
\begin{eqnarray*}
M(z,y)=R(z,y)- R_0(z,y)= \bar{M}(\bar{\Phi}^{-1}(z),\bar{\Phi}^{-1}(y))+ R_0(\bar{\Phi}^{-1}(z),y) - R_0(z,y) \ .
\end{eqnarray*}
Then using \eqref{2}, \eqref{3}, \eqref{Stimaasint} and \eqref{6} we find that 
\begin{eqnarray}\label{stimafunz}
|M(z,y)|&\le& C|z-y|^{\alpha+2-n} +\frac{C}{\rho_0^{\alpha}}\|\nabla R_0(\cdot,y) \|_{L^{\infty}(B'_{\rho_1}(0))}|z-\bar{\Phi}^{-1}(z)|\le\nonumber \\
&\le& C|z-y|^{\alpha+2-n} +\frac{C'}{\rho_0^{\alpha}}|z-y|^{1-n}|z-y|^{1+\alpha}\le\nonumber \\&\le& \frac{C''}{\rho_0^{\alpha}}|z-y|^{2+\alpha-n} \ ,
\end{eqnarray}
 where $C''>0$ depends on the a-priori data only.

We now estimate the gradient of $M$. Let $z\in \bar{\Phi}(B'_{\frac{\rho_1}{4}}(0))$ such that $z=\bar{\Phi}(\zeta)$ and let $h=|\zeta-y|$. The following interpolation inequality holds 
\begin{eqnarray}\label{interpol}
\|\nabla_{\zeta} \bar{M}(\cdot,y)\|_{L^{\infty}(B'_{\frac{\rho_1}{4}}(0))}\le C \|\bar{M}(\cdot,y) \|_{L^{\infty}(B'_{\frac{\rho_1}{4}}(0))} ^{\frac{\alpha}{1-\alpha}}|\nabla_{\zeta} \bar{M}(\cdot,y)|_{\alpha, B'_{\frac{\rho_1}{4}}(0)}^{\frac{1}{1-\alpha}} \ \ ,
\end{eqnarray}
where $C>0$ depends on the a-priori data only and 

\begin{eqnarray*}
|\nabla_{\zeta} \bar{M}(\cdot,y)|_{\alpha, B'_{\frac{\rho_1}{4}}(0)}
=\sup_{\substack {\zeta,\zeta'  \in B'_{\frac{\rho_1}{4}(O)} \\ \zeta\neq \zeta' }}\frac{|\nabla_{\zeta} \bar{M}(\zeta,y)-\nabla_{\zeta} \bar{M}(\zeta',y)|}{|\zeta-\zeta'|^{\alpha}} \ .
\end{eqnarray*}
By the H\"{o}lder continuity of $\nabla_{\zeta} \bar{R}$ and also of $\nabla_{\zeta}R_0$  we have that

\begin{eqnarray}\label{in}
|\nabla_{\zeta} \bar{M}(\cdot,y)|_{\alpha, B'_{\frac{h}{4}}(0)}\le&& \frac{C}{{h}^{\alpha}}\left(\|\nabla_{\zeta} \bar{R}(\cdot,y) \|_{L^{\infty}(B'_{\frac{h}{2}}(0))}+\|\nabla_{\zeta} {R_0}(\cdot,y) \|_{L^{\infty}(B'_{\frac{h}{2}}(0))}\right)\ \nonumber \\
\le&& C h^{1-n-\alpha}  
\end{eqnarray}
where $C>0$ depends on the a-priori data only. 

Hence combining \eqref{stimafunz},\eqref{interpol} and \eqref{in} we get 
\begin{eqnarray}\label{1a}
|\nabla_z \bar{M}({\bar{\Phi}}^{-1}(z),y)|\le&& \frac{C}{\rho_0^{\frac{\alpha^2}{\alpha+1}}}|z-y|^{(2+\alpha-n)(\frac{\alpha}{1-\alpha})}|z-y|^{(1-\alpha-n)(\frac{1}{1-\alpha})}=\nonumber\\
=&&\frac{C}{\rho_0^{\frac{\alpha^2}{\alpha+1}}}|z-y|^{1-n +\frac{\alpha^2}{\alpha+1}}
\end{eqnarray}
On the other hand we have 
\begin{eqnarray}\label{1b}
&&|\nabla_z R_0({\bar{\Phi}}^{-1}(z),y) - \nabla_z R_0(z,y)|\le\nonumber\\ 
&&\le|(D{\bar{\Phi}}^{-1}(z)^T-I)\nabla R_0(\cdot,y)|_{{\bar{\Phi}}^{-1}(z)} | + |\nabla R_0(\cdot,y)|_{{\bar{\Phi}}^{-1}(z)}- \nabla_z R_0(z,y)|\le\nonumber\\
&&\le\frac{C}{{\rho_0}^{\alpha}}\|\nabla R_0(\cdot,y) \|_{L^{\infty}(B'_{\rho_1}(0))}|z-{\bar{\Phi}}^{-1}(z)|+ |\nabla R_0(\cdot,y)|_{\alpha,B'_{\rho_1}(0) }|{\bar{\Phi}}^{-1}(z)-z |^{\alpha}\le\nonumber \\
&&\le\frac{C}{{\rho_0}^{{\alpha}^2}}|z-y|^{1-n}|z-y|^{1+\alpha} + \frac{C}{{\rho_0}^{{\alpha}^2}}|z-y|^{-\alpha+1-n}|z-y|^{(1+\alpha)\alpha}\le \nonumber\\
&&\le\frac{C}{{\rho_0}^{{\alpha}^2}}|z-y|^{1-n + \alpha^2}\ ,
\end{eqnarray}
where $C>0$ depends on the a-priori data only. Thus by \eqref{1a} and \eqref{1b} we obtain the thesis. \end{proof}

\begin{proposition}\label{holder}
Let $\Sigma$ be a crack satisfying the a-priori assumption stated above. Let $x\in \Omega\setminus\partial\Sigma$ and let $y\in \Omega\setminus \Sigma$ such that $|x-y|\le 2r_0$. Then there exist constants $C,\widetilde{\alpha}>0$ depending on the a-priori data only such that 
\begin{eqnarray}
|\nabla_x R(x,y)|\le C \left(  \frac{r^{\widetilde{\alpha}-1}}{|x-y|^{n-2+\widetilde{\alpha}}} \right) \ ,
\end{eqnarray}
where $r=\frac{1}{4}\min\{\mbox{dist}(x,\partial \Sigma), |x-y|\}$. 
\end{proposition}

\begin{proof} Let $v\in H^1(\Omega\setminus\overline{\Sigma})$ be the positive solution to \eqref{Positiva} introduced in Lemma \ref{soluzionepositiva}. We observe that $\displaystyle{N(x,y)=\frac{R(x,y)}{v(x)}}$ solves
\begin{equation*}\label{N}
\left\{
\begin{array}
{lcl}
\mbox{div}_x(v^2(x) {\nabla}_x N(x,y))=- \delta(x-y)\ ,&& \mbox{in $\Omega\setminus\overline{\Sigma}$ ,}
\\
\partial_{\nu}N(x,y)=0 \ , && \mbox{on $\Sigma$ .}
\end{array}
\right.
\end{equation*}

Let now $x$ be a point in $\Omega\setminus\partial\Sigma$ and let $y\in \Omega\setminus \Sigma$.

Let $r$ be the radius defined as follows $r=\frac{1}{4}\min\{\mbox{dist}(x,\partial \Sigma), |x-y|\}$. Without loss of generality we may assume that $x\in \overline{\Omega_+}$. Let $x'\in B_{r}(x)\cap \overline{\Omega_+}$ such that $|x-x'|\le\frac{r}{2}$.

By a change of variable argument we can deduce from Corollary 2.14 in  \cite{KePi} the following H\"{o}lder continuity property of the Neumann function 
\begin{eqnarray}\label{hol}
{{|N(x,y)-N(x',y)|\le C \frac{|x-x'|^{\widetilde{\alpha}}}{|x-y|^{n-2+\widetilde{\alpha}}+ |x'-y|^{n-2+\widetilde{\alpha}}} }} ,
\end{eqnarray}
where $C,\widetilde{\alpha}>0$ depends on the a-priori data only.

Moreover, being $|x'-y|\ge \frac{7}{8}|x-y|$ we can deduce that
\begin{eqnarray}\label{hol2}
|N(x,y)-N(x',y)|\le C \frac{|x-x'|^{\widetilde{\alpha}}}{|x-y|^{n-2+\widetilde{\alpha}}} \ ,
\end{eqnarray}
 up to a possible replacing of the constant $C$ in \eqref{hol}.
 
 Let us now consider $\bar{x}\in \partial B_r(x)\cap \overline{\Omega_+}$ and let $x''\in B_{\frac{r}{2}}(x)\cap\overline{\Omega_+}$, then by the above estimate and by observing that 
 $|\bar{x}-y|\ge \frac{3}{4}|x-y|$, we have that there exists a constant $C>0$ depending on the a-priori data only such that 
 \begin{eqnarray}\label{disegtr}
 |N(x'',y)|&&\le |N(x'',y)- N(\bar{x},y)| + |N(\bar{x},y)|\le \nonumber \\
 &&\le C \left(\frac{|x''-\bar{x}|^{\widetilde{\alpha}}}{|x-y|^{n-2+\widetilde{\alpha}}} + |{x}-y|^{2-n} \right) \ .
 \end{eqnarray}
 Moreover, by the following local bound for the gradient, we have that there exists a constant $C>0$ depending on the a priori data only such that 
 \begin{eqnarray}\label{locbound}
 \|\nabla_x N(\cdot,y)\|_{L^{\infty}(B_{\frac{r}{4}}(x) \cap \overline{\Omega_+} )}\le \frac{C}{r} \|N(\cdot,y)\|_{L^{\infty}(B_{\frac{r}{2}}(x)\cap\overline{\Omega_+})}\
 \end{eqnarray}
 By combining \eqref{disegtr} and \eqref{locbound} we find
 \begin{eqnarray}
 \|\nabla_x N(\cdot,y)\|_{L^{\infty}(B_{\frac{r}{4}}(x) \cap \overline{\Omega_+} )} \le C\left( \frac{r^{\widetilde{\alpha}-1}}{|x-y|^{n-2+\widetilde{\alpha}}} + \frac{r^{-1}}{|x-y|^{n-2}}\right) \ .
 \end{eqnarray}
 Next, being $|x-y|\le 2r_0$ we can find a constant $C$ depending on the a-priori data only such that 
 
 \begin{eqnarray}
 \|\nabla_x N(\cdot,y)\|_{L^{\infty}(B_{\frac{r}{4}}(x) \cap \overline{\Omega_+} )} \le C\frac{r^{\widetilde{\alpha}-1}}{|x-y|^{n-2+\widetilde{\alpha}}}  \ .
 \end{eqnarray}
 
Finally by the formal computation 
\begin{eqnarray}
\nabla_x R(x,y)= \nabla_x N(x,y)v(x) + \nabla v(x) N(x,y)
\end{eqnarray}
and by analogous arguments of those applied above, the thesis follows. \end{proof}

Let us consider $O\in \Sigma_1\cap \partial V_{l_1}$ the point in \eqref{puntodistanza}.
We introduce a point $O'\in \Sigma_1$ which is defined as follows by distinguishing two cases. 
\begin{itemize}
\item If $O\in \Sigma_1$ is such that $\mbox{dist}(O,\partial\Sigma_1)<\frac{d_{l_1}}{4}$, then we consider a point $O'\in \Sigma_1$ so that $\mbox{dist}(O,O')=\frac{d_{l_1}}{2}$. It follows that $\mbox{dist}(O',\partial \Sigma_1)\ge \frac{d_{l_1}}{4}$ and $\mbox{dist}(O',\Sigma_2)\ge \frac{d_{l_1}}{2}$. 
\item If $O\in \Sigma_1$ is such that $\mbox{dist}(O,\partial\Sigma_1)\ge\frac{d_{l_1}}{4}$ then we set $O'=O$.
\end{itemize}

\begin{proof}[Proof of Proposition \ref{stimealtobasso}.]
We begin by recalling \eqref{eg}
and we write
\begin{eqnarray}\label{diff}
|f(y,y)|\ge |S_{\Sigma_1}(y,y)|-|S_{\Sigma_2}(y,y)| \ .
\end{eqnarray}

First, we consider the term $S_{\Sigma_1}(y,y)$.  We fix a radius $\rho=\min\{cd_{l_1},c_0d_{l_1}^p\} $ and we observe that 

\begin{eqnarray}\label{Ssigma}
|S_{\Sigma_1}(y,y)|&\ge&\left|\int_{(\Sigma_1\setminus\Sigma_2)\cap B_{\rho}(O')}R_1^+(\cdot,y)\partial_{\nu_1}R_2^+(\cdot,y)d\sigma\right| -\nonumber\\
& -&\int_{(\Sigma_1\setminus\Sigma_2)\cap B_{\rho}(O')}|R_1^-(\cdot,y)\partial_{\nu_1}R_2^-(\cdot,y)d\sigma|  -\nonumber\\
&-& \int_{(\Sigma_1\setminus\Sigma_2)\setminus B_{\rho}(O')}|[R_1(\cdot,y)]_1\partial_{\nu_1}R_2(\cdot,y)|  d\sigma - \nonumber \\
&-&\int_{(\Sigma_1\setminus\Sigma_2)}|R_2(\cdot,y)[\partial_{\nu_1}R_1(\cdot,y)]_1|  d\sigma - \nonumber\\
&-& \int_{\Sigma_1\cap\Sigma_2}|[R_2(\cdot,y)\partial_{\nu_1}R_1(\cdot,y)]|  d\sigma
\end{eqnarray}

Let $\Gamma(x,y)$ be the fundamental solution to the Laplace equation. Then we have that for any $x\in B_{\rho}(O')$ and for any $y\in B_{\frac{\rho}{2}}(O')$ 
\begin{equation*}\
\left\{
\begin{array}
{lcl}
\Delta_{x}(R_2(x,y)-\Gamma(x,y))=0\ ,&& \mbox{in $B_{\rho}(O')$ ,}
\\
|R_2(x,y)-\Gamma(x,y)||_{\partial B_{\rho}(0')}\le C \rho^{2-n} \, 
\end{array}
\right.
\end{equation*}
where $C$ is a constant depending on the a-priori data only. 

By the maximum principle for harmonic functions we get 
\begin{eqnarray}
|R_2(x,y)-\Gamma(x,y)|\le C \rho^{2-n}\ \ \forall \ x \in B_{\rho}(O'), \ \forall \ y \in B_{\frac{\rho}{2}}(O') \ .
\end{eqnarray}
By standard gradient estimates we have that 
\begin{eqnarray}
|\nabla_x R_2(x,y)-\nabla_x \Gamma(x,y)|\le C \rho^{1-n}  \  \forall \ x \in B_{\frac{\rho}{2}}(O'), \ \forall \ y \in B_{\frac{\rho}{2}}(O'). 
\end{eqnarray}
Moreover observing that 
\begin{eqnarray}
&&R_1(x,y)= R_0(x,y) + (R_1(x,y)- R_0(x,y)) \ \\
&&\nabla_x R_2(x,y)= \nabla_x \Gamma(x,y) + (\nabla_x R_2(x,y)-\nabla_x \Gamma(x,y))
\end{eqnarray}
we have that 
\begin{eqnarray}
&&\left|\int_{(\Sigma_1\setminus\Sigma_2)\cap B_{\frac{\rho}{2}}(0')}R_1^+(x,y)\partial_{\nu(x)} R^+_2(x,y)d\sigma(x)\right|\ge \nonumber \\
&& \ge\left|\int_{(\Sigma_1\setminus\Sigma_2)\cap B_{\frac{\rho}{2}}(O')}R_0(x,y)\partial_{\nu(x)}\Gamma(x,y)d\sigma(x)\right|-\nonumber\\ 
&&- \int_{(\Sigma_1\setminus\Sigma_2)\cap B_{\frac{\rho}{2}}(0')}|R_0(x,y)\partial_{\nu(x)}(R^+_2(x,y)-\Gamma(x,y))|d\sigma(x) -\nonumber\\
&&-\int_{(\Sigma_1\setminus\Sigma_2)\cap B_{\frac{\rho}{2}}(0')}|(R_1^+(x,y)-R_0(x,y))\partial_{\nu(x)}\Gamma(x,y)|d\sigma(x)- \nonumber\\
&&-\int_{(\Sigma_1\setminus\Sigma_2)\cap B_{\frac{\rho}{2}}(0')}|(R_1^+(x,y)-R_0(x,y))\partial_{\nu(x)}(R^+_2(x,y)-\Gamma(x,y))|d\sigma(x)\nonumber
\end{eqnarray}

Let $\Phi \in \ C^{1,\alpha}(B_{\frac{\rho_0}{4M}}(O'),\mathbb{R}^n)$ be the change of coordinates introduced in Theorem \ref{intbyparts}, then we have that 
$\nu(x)=(0,\dots,0,1)+ \mathcal{O}(|x'|^{\alpha})$ and hence
\begin{eqnarray*}
&&\left|\int_{(\Sigma_1\setminus\Sigma_2)\cap B_{\frac{\rho}{2}}(O')}R_0(x,y)\partial_{\nu(x)}\Gamma(x,y)d\sigma(x)\right|\ge \hat{c_1}\int_{B'_{\frac{\rho}{16M}}(\Phi^{-1}(0'))}h|\xi'-\eta|^{2-2n}d\xi'+ \\
&& - \hat{c_2} \int_{B'_{\frac{\rho}{16M}}(\Phi^{-1}(0'))}h|\xi'-\eta|^{2-2n}|\xi'|^{\alpha}d\xi'\ge C_1 h^{2-n} - C_2 h^{2-n+\alpha}
\end{eqnarray*}
where $\hat{c_1},\hat{c_2},C_1,C_2>0$ are constants depending on the a-priori data only and  $\xi=(\xi',\xi_n)$, $\xi=\Phi^{-1}(x),\ \eta=\Phi^{-1}(y)$.

We now consider the second term on the right hand side of \eqref{Ssigma}. We have that 
\begin{equation*}\
\left\{
\begin{array}
{lcl}
\Delta_{x}(R_1(\cdot,y))=0\ ,&& \mbox{in $B_{\rho}(O')\cap \Omega^-$ ,}
\\
|R_1(\cdot,y)||_{\partial B_{\rho}(O')\cap \Omega^-}\le C \rho^{2-n} \, 
\\
\partial_{\nu^-}R^-(\cdot, y) - \gamma^-(\cdot) R^-(\cdot, y)=0 \ , && \mbox{in $B_{\rho}(O')\cap \Sigma^-$ .}
\end{array}
\right.
\end{equation*}
Hence by the weak maximum principle we have that 
\begin{eqnarray}
|R_1(\cdot,y)|\le C  {\rho}^{2-n} \ \ \ \mbox{in} \ \  B_{\rho}(O')\cap \Omega^-,
\end{eqnarray}
where $C>0$  is a constant depending on the a priori data only.

 Then, by the asymptotic formulas \eqref{asin1}, \eqref{asin2}, we get
 
 \begin{eqnarray}
&&\left|\int_{(\Sigma_1\setminus\Sigma_2)\cap B_{\frac{\rho}{2}}(O')}R_1^+(x,y)\partial_{\nu(x)} R^+_2(x,y)d\sigma(x) \right|-\nonumber\\
&&- \int_{(\Sigma_1\setminus\Sigma_2)\cap B_{\rho}(O')}|R_1^-(x,y)\partial_{\nu(x)}R_2^-(x,y)|d\sigma(x)\ge \nonumber\\
&&\ge  C_1 h^{2-n} - C_2 h^{2-n+\alpha} - \int_{(\Sigma_1\setminus\Sigma_2)\cap B_{\frac{\rho}{2}}(O')} C_3|x-y|^{2-n}d_{l_1}^{1-n}d\sigma(x) -\nonumber	\\
&&-\int_{(\Sigma_1\setminus\Sigma_2)\cap B_{\frac{\rho}{2}}(O')}\frac{C_4}{{d_{l_1}}^{\alpha}}|x-y|^{2-n+\alpha}|x-y|^{1-n}d\sigma(x)-\nonumber\\
&& -\int_{(\Sigma_1\setminus\Sigma_2)\cap B_{\frac{\rho}{2}}(O')}\frac{C_5}{{d_{l_1}}^{\alpha}}|x-y|^{2-n+\alpha}d_{l_1}^{1-n}d\sigma(x) - \nonumber\\
&&- \int_{(\Sigma_1\setminus\Sigma_2)\cap B_{\rho}(O')} C_6{{d_{l_1}}}^{2-n}|x-y|^{1-n}d\sigma(x)\ .
\end{eqnarray}

After straightforward calculation we observe that up to choosing the constant $c_0$ (in the definition of $h$ and $\rho$) sufficiently small, we have that 
 \begin{eqnarray*}
&&\left|\int_{(\Sigma_1\setminus\Sigma_2)\cap B_{\frac{\rho}{2}}(O')}R_1^+(x,y)\partial_{\nu(x)} R^+_2(x,y)d\sigma(x)\right|-\nonumber\\
&&-\int_{(\Sigma_1\setminus\Sigma_2)\cap B_{\rho}(O')}|R_1^-(x,y)\partial_{\nu(x)}R_2^-(x,y)|d\sigma(x) \ge c_1 h^{2-n} \ \  \ \ \ \ \ \  \ \ \
\end{eqnarray*}
where $c_1>0$ is a constant depending on the a-priori data only.  

We estimate the third term on the right hand side of \eqref{Ssigma}. By the asymptotic estimate \eqref{asin1}, we have that 
\begin{eqnarray}
|{R_1}^{\pm}(\cdot, y)|\le C |x-y|^{2-n}\le C \left|d_{l_1}^p-h\right|^{2-n}\  \ \ \mbox{on}\ (\Sigma_1\setminus\Sigma_2)\setminus B_{\rho}(O')\nonumber
\end{eqnarray}
where $C>0$ is a constant depending on the a priori data only. 

Moreover, by Proposition \ref{holder}, we infer that, on $(\Sigma_1\setminus\Sigma_2)\setminus B_{\rho}(O')$
\begin{eqnarray}
|\partial_{\nu_1}R_2(x,y)|\le C\left(\mbox{dist}(x,\partial \Sigma_2)^{\widetilde{\alpha}-1}|x-y|^{2-n-\widetilde{\alpha}} + |x-y|^{1-n} \right) \nonumber \ .
\end{eqnarray}
Hence by the integrability of $\mbox{dist}(x,\partial \Sigma_2)$ over $\Sigma_1\setminus\Sigma_2$ we deduce that 
\begin{eqnarray}
\int_{(\Sigma_1\setminus\Sigma_2)\setminus B_{\rho}(O')}|[R_1(x,y)]_1\partial_{\nu(x)}R_2(x,y)|  d\sigma(x)\le C\left|d_{l_1}^p-h\right|^{3-2n}
\end{eqnarray}
where $C>0$ is a constant depending on the a priori data only.
Finally, by \eqref{asin1} and by the Robin boundary condition we get  
\begin{eqnarray}
&&\int_{\Sigma_1\setminus\Sigma_2}|R_2(x,y)[\partial_{\nu(x)}R_1(x,y)]_1|  d\sigma(x) + 
\int_{\Sigma_1\cap\Sigma_2}|[R_2(x,y)\partial_{\nu_1}R_1(x,y)]|  d\sigma(x)\le \nonumber\\ && \le C \left|d_{l_1}^p-h\right|^{4-2n}\nonumber \ .
\end{eqnarray}
Gathering together the above estimates we get
\begin{eqnarray}
|S_{\Sigma_1}(y,y)|\ge c_1 h^{2-n} - \widetilde{c_1}|d_{l_1}^p -h|^{3-2n}.
\end{eqnarray}

The  upper bound 
\begin{eqnarray}
|S_{\Sigma_2}(y,y)|\le \widetilde{c_2}|d_{l_1}^p -h|^{3-2n}.
\end{eqnarray}
follows along the same lines of the arguments above.

Combining the last two inequalities and \eqref{diff} we conclude the proof.

\end{proof}

\end{section}

\end{document}